\documentclass[10pt,a4paper]{amsart}
\RequirePackage[OT1]{fontenc}
\RequirePackage{amsthm,amsmath,amssymb,mathrsfs, enumerate}
\RequirePackage[colorlinks,citecolor=blue,urlcolor=blue]{hyperref}
\RequirePackage{hyphenat}
\newtheorem{theorem}{Theorem}[section]
\newtheorem{corollary}[theorem]{Corollary}
\newtheorem{lemma}[theorem]{Lemma}
\newtheorem{prop}[theorem]{Proposition}

\newtheorem{definition}[theorem]{Definition}

\newtheorem{probl}[theorem]{Problem}
\newtheorem{remark}[theorem]{Remark}

\newcommand{\Ii}{1\!\!1}
\newcommand{\be}{\begin{equation}}
\newcommand{\ee}{\end{equation}}

\usepackage{graphicx}
\def\RR{\mathbb{R}}
\def\NN{\mathbb{N}}
\def\CC{\mathbb{C}}

\def\x{{\bf x}}
\def\C{{\mathcal C}}

\title[Quasi-analyticity and determinacy of the full moment problem]{Quasi-analyticity and determinacy of the full moment problem from finite to infinite dimensions}
\author[M. Infusino]{Maria Infusino$^{*}$}
\address{Fachbereich Mathematik und Statistik, Universit\"at Konstanz,\newline
\indent Universit\"atstrasse, 10, Konstanz, 78457, Germany}
\email{infusino.maria@gmail.com}
\thanks{$^{*}$ Supported by a Marie Curie fellowship of the Istituto Nazionale di Alta Matematica (INdAM)}
\keywords{Moment problem; determinacy; quasi-analytic class; infinite-dimensional moment problem; realizability; nuclear space; convex regularization; log-convexity}
\subjclass[2010]{44A60, 26E10, 47B25, 47A70, 28C05, 28C20}
\begin{document}
\begin{abstract}
This paper is aimed to show the essential role played by the theory of quasi-analytic functions in the study of the determinacy of the moment problem on finite and infinite-dimensional spaces. In particular, the quasi-analytic criterion of self-adjointness of operators and their commutativity are crucial to establish whether or not a measure is uniquely determined by its moments. Our main goal is to point out that this is a common feature of the determinacy question in both the finite and the infinite-dimensional moment problem, by reviewing some of the most known determinacy results from this perspective. We also collect some properties of independent interest concerning the characterization of quasi-analytic classes associated to log-convex sequences.
\end{abstract} 
\maketitle
\section*{Introduction}
Among the numerous aspects of the moment problem, the so-called \emph{determinacy question} is certainly one of the most investigated but still far from being completely solved. The moment problem asks whether a given sequence of numbers is the sequence of moments of some non-negative measure with fixed support. If such a measure is unique, then the moment problem is said to be \emph{determinate}. Therefore, the determinacy question is to find under which conditions a non-negative measure with given support is completely characterized by its moments.

In this paper we give an overview about how the concept of quasi-analyticity enters in the study of the determinacy question. As spectral theory and moment theory developed in parallel, the determinacy proofs which can be found in literature often seem circular. We review some of them showing the essential role played by quasi-analyticity techniques.

The moment problem was originally posed in a finite-dimensional setting (see e.g.\!~\cite{Akh65, Sh-Tam43}). More precisely, in the \emph{multivariate power moment problem} the starting multisequence $(m_\alpha)_{\alpha\in\NN_0^d}$ consists of real numbers and the support of the measure is assumed to be a subset $K\subseteq\RR^d$, where $d\in\NN$. However, at an early stage, this problem has also been generalized to the case of infinitely many variables (see e.g.~\cite{BeKo88} for more details on this topic). This abstract formulation of the moment problem is actually very useful in many applications related to the analysis of many-body systems, e.g. in statistical mechanics, spatial ecology, etc. In this setting, each $m_n$ in the starting sequence $(m_n)_{n\in\NN_0}$ is an element of the tensor product of $n$ copies of a certain infinite-dimensional space (e.g.\! for each $n$, $m_n$ is a generalized function of $n$ variables in~$\RR^d$) and the support of the measure is assumed to be a non-linear subset of this space (examples of supports are the set of all $L^2$ functions, the cone of all non-negative generalized functions, the set of all signed measures). 

This paper attempts to show that, regardless of the dimension of the setting in which the moment problem is posed, quasi-analyticity theory gives in some sense the best possible general sufficient determinacy conditions. In the literature, there are different approaches to the investigation of the determinacy question for concrete cases in the finite-dimensional context (see ~\cite{PutSchm08} for a thorough overview). For instance, the first known determinacy conditions for the one-dimensional moment problem were obtained through techniques involving continued fractions (see e.g.\! ~\cite{Hamb19, Hamb20, Perr50, Stiel1894}) or using density conditions of polynomials (see e.g.\! ~\cite{Akh65, Ber96, BerChrist81, Riesz23, Sh-Tam43, Stone32}). The determinacy of the higher-dimensional moment problem is still less understood than the one-dimensional case. However, a number of sufficient multivariate determinacy conditions were developed by using polynomial and rational approximation (see e.g.\! ~\cite{Ber95,  Fug83, Naim46, Pet82, PutScheid06}).

The link between quasi-analyticity and determinacy has been known since the early days of the moment theory. In 1926, Carleman first applied quasi-analyticity to the study of the determinacy of the univariate moment problem (see ~\cite{Carl26}). More precisely, he proved that if the moment sequence $(m_n)_{n\in\NN_0}$ of a non-negative measure $\mu$ supported on $\RR$ fulfills the so-called \emph{Carleman condition}, i.e. $\sum_{n=1}^\infty m_{2n}^{-\frac{1}{2n}} =\infty$, then there is no other measure having the same moment sequence as $\mu$. His main idea was to exploit the quasi-analyticity of a certain integral transform, which intrinsically contains the moment data, to get the determinacy of the classical Hamburger moment problem (c.f.~Theorem~\ref{UniqOneDimHamb}).

The concept of quasi-analytic function was first introduced by Borel, who observed that there is a larger class of functions, than merely the analytic functions, having the property to be completely determined only by their value and the values of their derivatives at a single point (see e.g.\!~\cite{Bor17}). Motivated by the theory of partial differential equations, Hadamard proposed the problem to give necessary and sufficient conditions bearing on a sequence $(m_n)_{n\in\NN_0}$ such that the class of all infinitely differentiable functions whose $n-$th derivative is bounded by $m_n$, for each $n\in\NN_0$, is quasi-analytic,~\cite{Hadam12}.  Denjoy was the first to provide sufficient conditions~\cite{Denj21} and then Carleman, generalizing Denjoy's theorem, gave necessary and sufficient conditions. Carleman's treatise~\cite{Carl26} threw a new light on the theory of quasi-analytic functions, revealing its important role in the study of the moment problem. His ideas inspired a large series of subsequent works about quasi-analyticity criteria for functions in one variable (see e.g.\!~\cite{Ban46, Coh68, Mand42, Ostr29}) and for multivariate functions (see e.g.\!~\cite{Boch50, BochTayl39, Ehr70}).

Both in the higher-dimensional case and in the infinite-dimensional one, the operator theoretical approach is a powerful method to get not only determinacy conditions but also to guarantee the existence of a solution to the moment problems. Actually, there has always been a mutual exchange between spectral theory and moment problem, since the results in moment theory often served as a starting point for new advances in the theory of operators. The quasi-analyticity again enters in a crucial way in the analysis of the moment problem from the operator theoretical point of view.

For the one-dimensional moment problem the basic ideas developed with the traditional methods of continued fractions and orthogonal polynomials can be retrieved by means of the spectral theory of self-adjoint extensions (see e.g.\!~\cite[Chapter~4]{Akh65}). In particular, the classical Hamburger and Stieltjes existence theorems and the relative uniqueness results due to Carleman can be obtained using this approach (see~\cite{Sim98}). Note that, in the one-dimensional moment problem, the theory of quasi-analytic functions only appears in the uniqueness part via the concept of quasi-analytic vector for an operator.

In contrast to the one-dimensional case, in higher dimensions, one does not know how to prove existence without uniqueness. In fact, in dimension $d\geq 2$, we need to use the spectral theorem for several essential self-adjoint operators and this requires that the involved operators pairwise strongly commute (i.e.\!~their resolutions of identity commute). In~\cite[Theorem 6]{Nuss65}, Nussbaum proved that the strong commutativity and the essential self-adjointness can actually be derived using again the concept of quasi-analytic vectors. The so-called \emph{multivariate Carleman condition} gives a condition for the existence of a total subset of quasi-analytic vectors for the considered operators directly in terms of the starting multisequence. This yields the existence of a unique measure solving the given moment problem (see~\cite[Theorem~10]{Nuss65}). Other similar but slightly weaker results were proved before Nussbaum's theorem, using the determinacy of certain $1-$sequences derived from the starting positive semidefinite~$d-$sequence (see e.g.\!~\cite{Dev57, Esk60, Sh-Tam43}). For more recent results about partial determinacy see~\cite[Section~5]{PutSchm08}.

Despite of the fact that it is unknown how to prove the existence of a solution to the moment problem on $\RR^d$ with $d\geq 2$ without involving its determinacy, it is instead possible to use partial determinacy to conclude the determinacy of a moment $d-$sequence. Petersen actually proved a general result of this kind, showing that if all the $d$ marginal measures of a measure $\mu$ on $\RR^d$ are determinate then $\mu$ is determinate, too (see~Theorem~\ref{PetersenThm}). Another determinacy result not involving existence is due to De Jeu, who has recently proved the uniqueness part of the moment problem on $\RR^d$ and on the positive octant $\RR^d_+$ by following Carleman's path without using spectral theory,~\cite{DeJeu03}.

The operator theoretical approach is also applicable to the infinite-dimensional moment problem. In fact, several infinite-dimensional moment problems have been investigated using the theory of generalized eigenfunction expansion for self-adjoint operators (see e.g.\!~\cite{Ber59, Ber68, Ber02, BeKo88, BS71, KosMit60, {S76}}). This approach is well developed for nuclear spaces in~\cite[Chapter 8]{Ber68} and~\cite[Vol. II, Chapter~5]{BeKo88}, and it is a generalization of the method presented by Krein in~\cite{Kr46, Kr48}. In these works, Krein used the so-called method of directed functionals for self-adjoint operators instead of the spectral projection theorem for an infinite family of strongly commuting self-adjoint operators given in~\cite[Vol.~I, Chapter 3, Section 2]{BeKo88}. Further different methods to solve the moment problem on nuclear spaces were introduced in 1975 in \cite{Bor-Yng75} and in \cite{Cha-Slink75}  (see also \cite{Heg75} and \cite[Section 12.5]{Schmu90}). These approaches are essentially based on Choquet theory and decompositions of positive definite functionals on a commutative nuclear *-algebras into pure states.

We describe in this paper the infinite-dimensional moment problem on the dual $\Omega'$ of a nuclear space $\Omega$, showing that the proof scheme used to get the existence of a unique solution to the moment problem on $\RR^d$ can be carried over in this case. In fact, thanks to a certain \emph{determining condition}, it is possible to show that the family of operators associated to the starting positive semidefinite sequence has a total subset of quasi-analytic vectors. Hence, they admit unique strongly pairwise commuting self-adjoint extensions by Nussbaum's result. Therefore, by the spectral theorem for infinitely many unbounded self-adjoint operators, there exists a unique measure on~$\RR^{\NN_0}$ representing those operators. It remains to show that this measure is actually concentrated on~$\Omega'$. Note that the determining condition is the correspondent of the multivariate Carleman condition in the infinite-dimensional case. However, the infinite-dimensionality involves additional layers such as the uniformity in the index, regularity properties and growth restrictions on the moments as functions.\\

Let us outline the contents of this paper. \\
In Section~\ref{Sec-qaFunctions}, we recall the notion of quasi-analytic class of infinitely differentiable functions on $\RR$ and we introduce the famous Denjoy-Carleman theorem. We also review some different versions of the Carleman condition known in literature, pointing out the role of the log-convexity in the proof of these results. In particular, we recall the technique of the convex regularization by means of the logarithm, which is important in solving the problem of the equivalence of quasi-analytic classes. In this context, we propose a proof of the Denjoy-Carleman theorem due to Mandelbrojt, which we found interesting since it is based on completely different methods than the classical ones employing holomorphic function theory.

In Section~\ref{Sec-FinitDimMP}, we aim to show how quasi-analytic classes enter in the determinacy of the finite-dimensional moment problem. In Subsection~\ref{Subs-1DimMP}, we give an alternative and simple proof of the Carleman uniqueness results for the Hamburger and the Stieltjes one-dimensional moment problems, which exploits the quasi-analyticity of a certain Fourier-Stieltjes transform (see proof of Theorem~\ref{UniqOneDimHamb}). Moreover, we recall the importance of the geometry of the support $K$ in the determinacy of the $K-$moment problem. In Subsection~\ref{Subs-dDimMP}, we focus on the determinacy of the $d-$dimensional moment problem with $d\geq 2$. We first introduce the so-called multivariate Carleman condition \eqref{MultCarlCondFunct} and show that it is sufficient for the determinacy of the Hamburger $d-$dimensional moment problem by using a result due to Petersen. Then we sketch the proof of a recent version of the Denjoy-Carleman theorem for quasi-analytic functions in several variables, which can be used to give an alternative proof of the uniqueness result. Finally, we outline the operator theoretical approach to the Hamburger $d-$dimensional moment problem developed by Nussbaum, stressing the points where quasi-analyticity is fundamental to prove not only the uniqueness but also the existence of the solution. We also mention a uniqueness result for the $d-$dimensional version of the Stieltjes moment problem.

In Section~\ref{InftyDimMP}, we present the moment problem on conuclear spaces. We first introduce a sufficient condition for the determinacy of the analogue of the Hamburger moment problem in this infinite-dimensional setting and we prove this uniqueness result without using spectral theoretical tools. Then we review the main results by Berezansky, Kondratiev and {\v{S}}ifrin about the existence and the uniqueness of a solution for the analogues of the Hamburger and the Stieltjes moment problems on conuclear spaces. We point out that, as in the finite-dimensional case, the existence of a solution cannot be proved without using the determinacy of the moment problem and we sketch the steps of this proof where quasi-analyticity enters.

In Appendix~\ref{App-LogconvQA}, we prove some results about log-convex sequences, which are useful in relation to the quasi-analyticity of the associated classes of functions. 

\section{Characterization of quasi-analytic classes of functions on $\RR$}\label{Sec-qaFunctions}
Let us recall the basic definitions and state some preliminary results concerning the theory of quasi-analytic functions on $\RR$. In the following, we denote by $\NN_0$ the set of all non-negative integers and by $\mathcal{C}^\infty(X)$ the space of all infinitely differentiable real valued functions defined on the topological space~$X$.
\begin{definition}[The class $C\{M_n\}$]\label{DefClass}\ \\
Given a sequence of positive real numbers $(M_n)_{n\in\NN_0}$, we define the class $C\{M_n\}$ as the set of all functions $f\in\C^\infty(\RR)$ such that 
$$\left\|D^n f\right\|_\infty\leq\beta_f B_f^n M_n,\quad\forall\, n\in\NN_0$$
where $D^n f$ is the $n-$th derivative of $f$, $\left\|D^n f\right\|_\infty:=\sup_{x\in \RR}\left|D^n f(x)\right|$, and $\beta_f$, $B_f$ are positive constants only depending on $f$.
\end{definition}
\begin{definition}[Quasi-analytic class]\label{qaClass}\ \\
The class $C\{M_n\}$ of functions on $\RR$ is said to be quasi-analytic if the conditions
$$f\in C\{M_n\},\,\,(D^n f)(0)=0,\quad\forall\, n\in\NN_0$$
imply that $f(x)=0$ for all $x\in \RR$.
\end{definition}
The definition above can be given replacing $(D^n f)(0)$ with $(D^n f)(x_0)$, where $x_0$ is any other given point in the domain of the function $f$. Note that the analytic functions on $\RR$ correspond to the class $C\{n!\}$. It is obvious from the previous definitions that the following holds.
\begin{prop}\label{ConstTimesCarl}\ \\
Let $({M_n})_{n\in\NN_0}$ be a sequence of positive real numbers. $C\{M_n\}$ is quasi-analytic if and only if for any positive constant $\delta$ the class $C\{\delta M_n\}$ is quasi-analytic.
\end{prop}

Recall that $C\{M_n\}$ and $C\{M_n'\}$ are said to be \emph{equivalent} if there exist two constants $a,b>0$ such that $a^nM_n\leq M_n'\leq b^nM_n$ for any $n\in\NN_0$. This means that every function of either of these two classes belongs also to the other. The problem of constructing a sequence $(M_n')_{n\in\NN_0}$ in a simple relationship with a given starting sequence $(M_n)_{n\in\NN_0}$ such that the corresponding classes of functions are equivalent has extensively been studied (for more details see~\cite{Mand52}). In particular, we introduce here the so-called \emph{convex regularization of by means of the logarithm}.  
\begin{definition}[Log-convexity]\label{LogConv1}\ \\
A sequence of positive real numbers $(M_n)_{n\in\NN_0}$ is said to be log-convex if and only if for all $n \geq 1 $ we have that
$M_n^2 \leq M_{n-1} M_{n+1}$.
\end{definition}
\begin{definition}[Convex regularization by means of the logarithm]\label{LogConvReg}\ \\
Let $(M_n)_{n\in\NN}$ be a sequence of positive real numbers with $\liminf\limits_{n\to\infty} M_n^{\frac 1n}=\infty$. Define for any $r \geq 1$ the function
$
T(r):= \max\limits_{n\in\NN} \frac{r^n}{M_n}.
$
The {\it convex regularization of $(M_n)_{n\in\NN}$ by means of the logarithm} is the sequence $(M_n^c)_{n\in\NN}$ defined by
\begin{equation}\label{eqqa17}
\ln(M_n^c):= \sup_{t\geq 0} \left( nt -\ln(T(e^t)) \right),
\end{equation}
or equivalently,
$
M_n^c :=  \sup\limits_{t\geq 0} \frac{e^{tn}}{T(e^t)} = \sup\limits_{r \geq 1} \frac{r^n}{T(r)}.
$
\end{definition}
Note that (\ref{eqqa17}) means that for any $t\geq 0$ the line
$
x \mapsto tx -\ln(T(e^t))
$
is not above any of the points $(\ln(M_n^c))_{1\leq n<e^t}$.  The convex regularized sequence by means of the logarithm is indeed the largest convex minorant (i.e. the convex envelope) of the function $n \mapsto \ln(M_n)$. This means that $(M_n^c)_{n\in\NN}$ is a log-convex sequence and that for any $n\in\NN$, $M_n^c\leq M_n$. Clearly, if $(M_n)_{n\in\NN}$ is log-convex then $M_n^c\equiv M_n$ for all $n\in\NN$.

This procedure allows to explicitly construct, starting from any sequence $(M_n)_{n\in\NN}$ of positive real numbers with $\liminf\limits_{n\to\infty} M_n^{1/n}=\infty$, a log-convex sequence $(M^c_n)_{n\in\NN}$ such that the classes $C\{M_n\}$ and $C\{M^c_n\}$ are equivalent (see~\cite[Theorem~6.5.III]{Mand52}). Therefore, if $\liminf\limits_{n\to\infty} M_n^{1/n}=\infty$, then the class $C\{M_n\}$ is quasi-analytic if and only if $C\{M^c_n\}$ is quasi-analytic.
\begin{remark}\label{RemarkLimInf} \ \\
When we deal with quasi-analytic classes the assumption $\liminf\limits_{n\to\infty} M_n^{\frac 1n}=\infty$ is not restrictive, but actually gives the only interesting case. In fact, if $\liminf\limits_{n\to\infty} M_n^{\frac 1n}=0$, then  $C\{M_n\}$ is equivalent to $C\{0\}$ (which contains only the constants) and if $0<\liminf\limits_{n\to\infty} M_n^{\frac 1n}<\infty$, then $C\{M_n\}$ is equivalent to $C\{1\}$ (see~\cite[Theorem~6.5.III]{Mand52}). In both cases, $C\{M_n\}$ is already quasi-analytic. 
\end{remark}

The problem to give necessary and sufficient conditions bearing on the sequence $(M_n)_{n\in\NN_0}$ such that
the class $C\{M_n\}$ is quasi-analytic was proposed by Hadamard in~\cite{Hadam12}.  Denjoy was the first to provide sufficient conditions for the quasi-analyticity of a class~\cite{Denj21}, but the problem was completely solved by Carleman, who generalized Denjoy's theorem and methods giving the first characterization of quasi-analytic classes in~\cite{Carl26}. Using the convex regularization by means of the logarithm, other conditions equivalent to Carleman's one were obtained.
\begin{theorem}[The Denjoy-Carleman Theorem]\label{DJ-C}\ \\
Let $(M_n)_{n\in\NN_0}$ be a sequence of positive real numbers. Then the following conditions are equivalent
\begin{enumerate}[(a).]
\item \label{QuasiAn} $C\{M_n\}$ is quasi-analytic,
\item\label{CarlGen} $\sum\limits_{n=1}^\infty \frac{1}{\beta_n}=\infty$ with $\beta_n:=\inf_{k\geq n}\sqrt[k]{M_k},$
\item\label{Carl} $\sum\limits_{n=1}^\infty \frac{1}{\sqrt[n]{M_n^c}}=\infty$,
\item \label{CondOstr}
$\int_1^\infty \frac{\ln(T(r))}{r^2} dr = \infty,
$
\item\label{CondMandel} $\sum\limits_{n=1}^\infty \frac{M_{n-1}^c}{M_n^c}=\infty$,
\end{enumerate}
where $({M}^c_n)_{n\in\NN}$ is the convex regularization of $({M}_n)_{n\in\NN}$ by means of the logarithm and for any $r \geq 1$ the function $T$ is given by
$
T(r):= \max\limits_{n\in\NN} \frac{r^n}{M_n}.
$
\end{theorem}

Condition \eqref{CarlGen} and \eqref{Carl} are due to Carleman, \cite{Carl26} (see also~\cite{Coh68} for a simple but detailed proof). Condition \eqref{CondOstr} was instead introduced by Ostrowski in~\cite{Ostr29}, who was also the first to provide a new proof of Carleman's theorem. Moreover,  Condition~\eqref{CondMandel} was independently given by Mandelbrojt and Bang in~\cite{Mand42} and~\cite{Ban46}, respectively. 

A very nice proof of the equivalence of the conditions \eqref{CarlGen}, \eqref{Carl}, \eqref{CondOstr} and~\eqref{CondMandel} is given in~\cite[Theorem~1.8.VII]{Mand52}. For its simplicity, let us just sketch the proof that \eqref{Carl} and \eqref{CondMandel} are equivalent. By Proposition~\ref{ConstTimesCarl}, we can assume w.l.o.g. $M_0=1$ and so easily derive from Proposition~\ref{LogConvexChar}~(b) that $M_{n-1}^c\leq \left(M_n^c\right)^{1-1/n}$. Hence, \eqref{CondMandel} implies \eqref{Carl}. The converse follows by Carleman's inequality, that is, by using
$ \sum_{n=1}^\infty (a_1\cdots a_n)^{\frac 1n}\leq e\sum_{n=1}^\infty a_n$ for $a_n:=M^c_{n-1}/{M^c_n}$.\\

To complete this section, we propose the proof of the equivalence of conditions \eqref{QuasiAn}, \eqref{CondOstr} and \eqref{CondMandel} given by Mandelbrojt in~\cite{Mand52} (see in particular~Theorem~4.1.III). In contrast to Denjoy, Carleman and Ostrowski, Mandelbrojt's proof is not based on the theory of holomorphic functions but only on some considerations relative to the average values of a real function.

The equivalence of \eqref{CondOstr} and \eqref{CondMandel} easily follows by the following
\begin{equation}\label{eqMandel}
\int_1^\infty \frac{\ln(T(r))}{r^2} dr=\ln(T(1))+1+\sum_{n=1}^\infty \frac{M_n^c}{M_{n+1}^c}.
\end{equation}
As mentioned before, a detailed proof of this equality can be found in~\cite[Theorem~1.8.VII]{Mand52}. We give here just an idea of this proof. For any $t\geq 0$, denote by $N(t)$ the greatest $n\in\NN$ such that $tn-\ln(T(e^t))=\ln M_n$. Then one can easily see that 
$\ln T(e^t)-\ln T(e^{t'})=\int_t^{t'}N(s)ds$, for any $t,t'\geq 0$.
Since $N$ is a piecewise continuous function and it is monotone increasing, we can denote by~$t_k$ the points in which the function $N$ has a jump. Using Definition~\ref{LogConvReg}, we get that ${M_n^c}/{M_{n+1}^c}=t_k$ for any $N(t_k)\leq n\leq N(t_{k+1})$. Combining these two results and making some further calculations, one finally gets~\eqref{eqMandel}. 

The main ingredient used by Mandelbrojt to prove that \eqref{CondOstr} and \eqref{CondMandel} are both necessary and sufficient for quasi-analyticity is the construction of an infinitely differentiable function with compact support which belongs to the class $C\{M_n\}$. Let us preliminarily sketch such a construction. \\

For a sequence $(\gamma_n)_{n\in\NN}$ of positive constants and a function $g$ Lebesgue integrable on $[-\gamma, \gamma]$, we define 
$$
M(\gamma_1,\!\ldots,\!\gamma_n; g)(x):=\frac{1}{2^n \gamma_1\cdots\gamma_n}\!\int_{-\gamma_1}^{\gamma_1}\!\!\!\!\!\cdots\!\int_{-\gamma_n}^{\gamma_n}\!\!\!g(x+t_1 + \cdots +t_n) dt_1 \ldots dt_n.
$$
Let $(\mu_n)_{n\in\NN}$ be a sequence of positive constants with $\sum_{n=1}^\infty\mu_n=:\mu<\infty$ and let $f$ be a Lebesgue integrable function supported on $[a,b]$. For any $n\geq 1$, we set 
\[
M_n(x):= M(\mu_1,\ldots,\mu_n; f)(x),
\]
which is obviously zero outside the interval $I_n:=[a-\sum_{i=1}^n\mu_i, b+\sum_{i=1}^n\mu_i]$ and whose value is independent of the order in which the quantities $\mu_1,\ldots,\mu_n$ are taken. This definition is recursive since $M_n(x) = M(\mu_n; M_{n-1})(x)$. For any $n\in\NN$, the function $M_n(x)$ is differentiable in $I_1$ with first derivative equal to
\[
(DM_n)(x) = \frac{1}{2\mu_1} M(\mu_2,\ldots, \mu_n; f(\cdot +\mu_1)-f(\cdot - \mu_1))(x).
\]
Thus, $M_n(x)$ has first derivative uniformly bounded in $n$ and so the family $(M_n(x))_{n\in\NN}$ is equally graded continuous in $[a-\mu, b+\mu]$. Therefore, it tends uniformly to a continuous function $\psi(x)$ and so for all $x \in [a-\mu,b+\mu]$
\[
\psi(x) = \frac{1}{2\mu_1} \int_{-\mu_1}^{\mu_1} M(\mu_2,\ldots,\mu_n,\ldots; f(\cdot+t))(x) dt.
\]
Hence, the function $\psi(x)$ is infinitely differentiable on $[a-\mu,b+\mu]$ and zero outside this interval.

\proof{\it of \eqref{QuasiAn}$\Rightarrow$\eqref{CondOstr}}\ \\
By Remark~\ref{RemarkLimInf}, we can directly assume that $\liminf\limits_{n\to \infty}M_n^{\frac 1n}=\infty$. 
W.l.o.g. we can take $M_0=1$, since $C\{M_n\}$ and $C\{\frac{M_n}{M_0}\}$ coincide by Proposition~\ref{ConstTimesCarl}. Suppose $\sum_{n=1}^\infty M^c_{n}/M_{n+1}^c < \infty $ and repeat the construction above for $\mu_n := M^c_{n-1}/M_n^c$ and $f := \Ii_{[-\mu,\mu]}$ with $\mu:=\sum_{n=1}^\infty \mu_n$. Then we get that the associated limit function $\psi$ is infinitely differentiable on $[-2\mu,2\mu]$ and zero outside. As a consequence, all the derivatives of the function $\psi$ are zero at $\pm 2 \mu$. Furthermore, 
$$
\psi(0)= \lim_{n \rightarrow \infty} \frac{1}{2^n \mu_1 \ldots \mu_n}
\int_{-\mu_1}^{\mu_1} \ldots \int_{-\mu_n}^{\mu_n} \Ii_{[-\mu,\mu]}(t_1 + \ldots t_n) dt_1\ldots dt_n =1
$$
and it is easy to show that
\[
\left|(D^n\psi)(x)\right| \leq \frac{1}{\mu_1\ldots \mu_n} = M_n^c\leq M_n.
\]
In conclusion, we constructed $\psi \in C\{M_n\}$ which is not quasi-analytic. \\
\endproof

To prove that \eqref{CondMandel} implies \eqref{QuasiAn} we need the following lemma.
\begin{lemma}\label{lemmino}\ \\
Let $(M_n)_{n\in\NN_0}$ a sequence of positive real numbers such that $C\{M_n\}$ is not quasi-analytic.
Then there exists an infinite differentiable function $\varphi$ on $[0,1]$ such that
\begin{enumerate}
\item $(D^n\varphi)(0)=0$ and $(D^n\varphi)(1)=0$, $\forall n \in \mathbb{N}_0$.
\item  $\forall n \in \mathbb{N}$ and $\forall x\in[0,1]$,
$
 \left|(D^n\varphi)(x)\right| \leq M_n.
$
\item $\varphi\geq 0$ on $[0,1]$.
\item $\varphi(1-x) = \varphi(x),\,\forall x\in[0,1]$.
\end{enumerate}
\end{lemma}

\proof\ \\
Since a class of functions is invariant under rescaling and translation, we can assume w.l.o.g. that the functions in $C\{M_n\}$ are defined on the interval $[0,1]$. As $C\{M_n\}$ is not quasi-analytic, there exists a non-zero function $f\in C\{M_n\}$ and a point $a\in[0,1[$ such that $f$ and all its derivatives vanish at $a$ but for any $\varepsilon >0$ the function $f$ is not identically zero on $[a,a+\varepsilon]$. For any $0\leq \alpha<1-a$, let us define for $x\in[0,1]$ $$f_1(x):=\int_0^{\alpha x+a}\int_0^tf(\tau)d\tau dt,$$ then $f_1$ is not identically zero on $[0,1]$ and all its derivatives vanish at $0$. Since $f\in C\{M_n\}$, there exists $c>0$ such that for any $n\in\NN_0$ and any $x\in[0,1]$, $|(D^nf)(x)|\leq c^{n}M_{n}$. Then $|(D^nf_1)(x)|\leq \alpha^nc^{n-2}M_{n-2}$, that is, $f_1\in C\{M_{n-2}\}$.

Let $f_2(x) := f_1(x-x^2)$, then $(Df_2)(x)=(1-2x)(Df_1)(x-x^2)$. By induction, it can be easily proved that for any $n\geq 2$ we have
\begin{align}\label{conti1}
\left|(D^nf_2)(x)\right| &\leq\sum_{k=0}^{\lfloor n/2 \rfloor} \frac{n^{2k}}{k!}\sup_{y\in [0,1]}\left|(D^{n-k}f_1)(y)\right|.
\end{align}
Using Taylor formula and the fact that all derivatives of $f_1$ vanish at zero, we obtain
\be\label{conti2}
(D^{n-k}f_1)(x) = \frac{1}{(k-1)!} \int_0^x (x-t)^{k-1} (D^nf_1)(t) dt.
\ee
By \eqref{conti1} and \eqref{conti2}, we get
$$\sup_{x\in [0,1]}|(D^nf_2)(x)| \leq \alpha^nc^{n-2}M_{n-2}\sum_{k=0}^\infty\frac{n^{2k}}{(k!)^2}\leq e^{2n}\alpha^nc^{n-2}M_{n-2}.$$
Hence, $f_2$ is in the same class of $f_1$ and vanishes with all its derivatives at $0$ and at $1$.

Let us consider the function $f_3(x) :=f_2(x)^2$. We can extend $f_2$ to a periodic even function, using that $f_2$ and all its derivatives coincide at the endpoints. Hence,
$$
f_2(x) = \frac{d_0}{2} + \sum_{q=1}^\infty (-1)^q d_q \cos(2\pi qx)\,\,\,\text{and}\,\,\,
|(D^nf_2)(x)| \leq (2\pi)^n \sum_{q=1}^\infty  |d_q| q^n,
$$
where
$
d_q := 2(-1)^q \int_0^1 f_2(x) \cos(2\pi qx)dx.
$
Integrating by parts $l$ times, we get that
\be\label{boundq}
|d_q | \leq  2 (2\pi q)^{-l} e^{2l} \alpha^lc^{l-2}M_{l-2} .
\ee
Using the binomial formula for the derivative and the H\"older inequality, we obtain that there exists $C>0$ such that
\[
|(D^nf_3)(x)| \leq (4\pi)^n C\sum_{q=1}^\infty |d_q|q^n .  
\]
Furthermore, by \eqref{boundq} for $l=n+2$, we get that
$$
|(D^nf_3)(x)| \leq L \left(\sum_{q=1}^\infty \frac{1}{q^{2}}\right)(2c\alpha e^2)^n  M_n,
$$
where
$L:=2C(2\pi)^{-2}(\alpha e^2)^2.$ If we choose $\alpha<\frac{1}{2ce^2}$, then the function we are looking for is given by $\varphi(x):=\frac{f_3(x)}{L\left(\sum\limits_{q=1}^\infty \frac{1}{q^{2}}\right)}$.\\
\endproof

\proof {\it of \eqref{CondMandel}$\Rightarrow$\eqref{QuasiAn}}\ \\
Let us show that if $C\{M_n\}$ is not quasi-analytic then 
$
\int_1^\infty \frac{\ln(T(r))}{r^2} dr < \infty.
$
By Remark~\ref{RemarkLimInf}, we can again assume that $\liminf\limits_{n\to \infty}M_n^{\frac 1n}=\infty$.
Let $f$ be a function on $[0,1]$ as given by Lemma~\ref{lemmino} and define
\[
F(z) := \int_0^1 e^{-xz} f(x)dx,\quad z\in\CC, 
\]
which is an entire function with $F(1)>0$. Using integration by parts $k$ times and the fact that $f$ vanishes with all its derivatives at $0$ and $1$, we get that
$
|F(z)| \leq \frac{M_k}{|z|^k}.
$
Since this holds for all $k\in\NN$ and $\liminf\limits_{n \to\infty} M_n^{1/n} = \infty $, we get
\be\label{boundF}
|F(z)| \leq \frac{1}{\max\limits_{k \in\NN}\frac{|z|^k}{M_k}}=\frac{1}{T(|z|)}.
\ee
Let $0<p<1$ and let us consider the circle $C_p$ given by the equation $\left|\frac{1-z}{z}\right|=p$. Using the Poisson integral formula and the properties of $F$, it is possible to prove that 
$$\frac{1}{2p\pi}\int_{C_p}\frac{\ln|F(z)|}{|z|^2}d|z|\geq\ln|F(1)|.$$
By using \eqref{boundF} in the latter equation, we get that
$$\frac{1}{2p\pi}\int_{C_p}\frac{\ln(T(|z|))}{|z|^2}d|z|\leq -\ln|F(1)|.$$
If we denote by $C_p^t$ the part of $C_p$ between the lines $Im(z)=-t$ and $Im(z)=t$ which contains the point $\frac{1}{p+1}$, then for large values of $t$ we have 
$$\frac{1}{2p\pi}\int_{C_p^t}\frac{\ln(T(|z|))}{|z|^2}d|z|\leq -\ln|F(1)|.$$
If $p\to 1$, then $C_p^t$ tends to the segment of the straight line $Re(z)=\frac 12$ with $-t<Im(z)<t$, which yields 
$$\frac{1}{\pi}\int_0^t\frac{\ln(T(r))}{\frac 14+r^2}dr\leq -\ln(F(1)).$$
As a consequence, the integral 
$\int_1^\infty\frac{\ln(T(r))}{r^2}dr<\infty.$\\
\endproof

\section{Uniqueness in the finite-dimensional moment problem}\label{Sec-FinitDimMP}
\subsection{The finite-dimensional moment problem}\ \\
Let $\RR[\x]$ be the algebra of all real polynomials with $d$ real variables and real coefficients. For $\alpha:=(\alpha_1,\dots,\alpha_d)\in\NN_0^d$ and $\x:=(x_1,\dots,x_d)\in\RR^d$, we define the following multi-index notation 
$\textbf{x}^{\alpha}:=x_1^{\alpha_1}\cdots x_d^{\alpha_d}$ (where $x_j^0:=1$) and $\left|\alpha\right|:=\alpha_1+\dots+\alpha_d$. Let $\mathcal{M}^*(\RR^d)$ be the collection of all non-negative Borel measures on $\RR^d$ such that $\x^\alpha\in L^1(\mu)$ for all $\alpha\in\NN_0^d$. 
\begin{definition}\ \\
Let $\mu\in\mathcal{M}^*(\RR^d)$ and $\alpha\in\NN_0^d$. The $\alpha^{\text{th}}-$\emph{moment} of $\mu$ is defined by
\begin{equation*}
	m_{\alpha}^{\mu}:=
	\int_{\RR^d}{\x^\alpha\,\mu(d\x)}=\int_{\RR^d}x_1^{\alpha_1}x_2^{\alpha_2}\cdots x_d^{\alpha_d}\mu(dx_1,dx_2,\dots,dx_d).
\end{equation*}
The multisequence $(m_{\alpha}^{\mu})_{\alpha\in\NN_0^d}$ is called \emph{moment sequence of $\mu$}. 
\end{definition}
Note that the set $\mathcal{M}^*(\RR^d)$ exactly consists of all the non-negative Borel measures on $\RR^d$ with finite moments of all orders. Given a closed subset $K\subseteq\RR^d$, we denote by $\mathcal{M}^*(K)$ the set of all measures in $\mathcal{M}^*(\RR^d)$ having support contained in~$K$. 

The $K-$moment problem asks to determine when a given multisequence is actually the moment sequence of some measure $\mu\in\mathcal{M}^*(K)$. 
\begin{probl}[Full $K-$moment problem]\label{KMP}\ \\
Let $m=(m_{\alpha})_{\alpha\in\NN_0^d}$ be a multisequence of real numbers and let $K\subseteq\RR^d$ be closed. Find a measure $\mu\in\mathcal{M}^*(K)$ such that
$m_{\alpha}=m_{\alpha}^{\mu}$ for all $\alpha\in\NN_0^d.$
\end{probl}
If such a measure $\mu$ does exist we say that the sequence $m$ is \emph{realized} by $\mu$ on $K$ and the measure $\mu$ is called \emph{realizing measure} on $K$. Note that we refer to this moment problem as \emph{finite-dimensional} since the dimension of the supporting set $K$ is finite. Recall that if $m$ is a finite sequence then the $K-$moment problem is called \emph{truncated}.

The statement of Problem~\ref{KMP} includes all the classical one-dimensional cases. In fact, if $d=1$, then we get
\begin{itemize}
\item The \emph{Hamburger moment problem} for $K=\RR$.
\item The \emph{Stieltjes moment problem} for $K=\RR_+$.
\item The \emph{Hausdorff moment problem} for $K=[0,1]$.
\end{itemize}
It is easy to see that the $K-$moment problem can be restated in terms of integral representation of linear functionals by introducing the so-called \emph{Riesz functional}.
\begin{definition}[Riesz' functional]\label{Lm}\ \\
Given $m=(m_{\alpha})_{\alpha\in\NN_0^d}$, we define the associated Riesz functional $L_m$ on $\RR[\x]$ by 
$
L_m(\x^{\alpha}):= m_{\alpha},
$ $\alpha\in\NN_0^d$.
\end{definition}

A necessary condition for a sequence of real numbers to be the moment sequence of some measure in $\mathcal{M}^*(\RR^d)$  is the following.
\begin{definition}[Positive semidefinite sequence]\label{PSD}\ \\
A sequence $m=(m_{\alpha})_{\alpha\in\NN_0^d}$ of real numbers is said to be \emph{positive semidefinite} if for any $n\in\NN$, $\alpha_1,\ldots,\alpha_n\in\NN_0^d$ and $\xi_1,\ldots,\xi_n\in\RR$, 
$$\sum_{j,l=1}^n m_{\alpha_j+\alpha_l}\xi_j\xi_l\geq 0,$$
or equivalently, for any $h\in\RR[\x]$,
$L_m(h^2)\geq 0.$
\end{definition}
In the case of the Hamburger moment problem, i.e.\! when $d=1$ and $K=\RR$, the positive semidefiniteness is also sufficient, but this is not true when $K=\RR^d$ with $d\geq 2$.

A measure $\mu\in\mathcal{M}^*(K)$ is called \emph{determinate} if any other measure $\nu\in\mathcal{M}^*(K)$ having the same moment sequence as $\mu$ is equal to $\mu$. Equivalently, a sequence of real numbers is called \emph{determining} on $K$ if there exists a unique non-negative measure in $\mathcal{M}^*(K)$ realizing $m$. In this case, the $K$-moment problem is also addressed as determinate.

\subsection{Determinacy conditions in the one-dimensional case}\label{Subs-1DimMP}\ \\
As far as we know, Carleman was the first to approach the determinacy question with methods involving quasi-analyticity theory. In fact, in his famous work of 1926, he proposed the following result which gives a sufficient condition for the uniqueness of the solution to the Hamburger moment problem (see ~\cite[Chapter~VIII]{Carl26}).
\begin{theorem}\label{UniqOneDimHamb}\ \\
Let $\mu,\nu\in\mathcal{M}^*(\RR)$ have the same moment sequence $m=(m_n)_{n\in\NN_0}$. If~$m$ is such that
\begin{equation}\label{CarlCond}
\sum_{n=1}^\infty \frac{1}{\sqrt[2n]{m_{2n}}} = \infty,
 \end{equation}
then $\mu=\nu$.
\end{theorem}
The original proof by Carleman makes use of the Cauchy transform of the two given measures. Here, we decided to propose a slightly different proof that uses the Fourier-Stieltjes transform but maintains the same spirit of Carleman's proof. In fact, the essential strategy of both proofs is to consider the transform of the difference of the two given measures and show that it belongs to the class $C\{\sqrt{m_{2n}}\}$, which can be proved to be quasi-analytic thanks to \eqref{CarlCond}. This directly leads to the fact that the two original measures coincide and so to the determinacy of the Hamburger moment problem for~$m$. Before giving our proof of Theorem~\ref{UniqOneDimHamb}, let us observe a useful property of the moment sequences.
\begin{remark}\label{RemarkMomentsLogConvex}\ \\
The log-convexity (see Definition \ref{LogConv1}) is a necessary condition for a sequence of positive numbers to be the absolute moment sequence of some non-negative measure defined on $\RR$. More precisely, if $\mu\in\mathcal{M}^*(\RR)$, then the sequence $(M_n)_{n\in\NN_0}$ of all its absolute moments, i.e.\! $M_n=\int_\RR{|x|^n\mu(dx)}$, is log-convex. 
In fact, by Cauchy-Schwarz's inequality, we have that for any $n\in\NN$
$$
M_n^2\leq \left(\int_\RR{|x|^{n-1}\mu(dx)}\right) \left(\int_\RR{|x|^{n+1}\mu(dx)}\right)\\
= M_{n-1} M_{n+1}.
$$
It directly follows that the sequence of all even moments of a measure $\mu\in\mathcal{M}^*(\RR)$, i.e.\! $m_{2n}=\int_\RR{x^{2n}\mu(dx)}=M_{2n}$, is also log-convex.
\end{remark}

\begin{proof}\emph{of Theorem \ref{UniqOneDimHamb}}\ \\
W.l.o.g. assume that all even moments of $\mu$ are positive. In fact, if $m_{2n}=0$ for some $n\geq 1$ then $supp(\mu)\subseteq \{x\in\RR: x^{2n}=0\}=\{0\}$ and thus, the unique realizing measure is $\mu=m_0\delta_0$. By Remark \ref{RemarkMomentsLogConvex}, the sequence of all even moments $(m_{2n})_{n\in\NN_0}$ is log-convex. Hence, the sequence $(\sqrt{m_{2n}})_{n\in\NN_0}$ is also log-convex and by assumption it satisfies \eqref{CarlCond}, which can be rewritten as $$\sum_{n=1}^\infty \frac{1}{\sqrt[n]{\sqrt{m_{2n}}}} = \infty.$$ Then, by Denjoy-Carleman's Theorem \ref{DJ-C}, the class $C\{\sqrt{m_{2n}}\}$ is quasi-analytic.

Let us consider the Fourier-Stieltjes transforms of $\mu$ and~$\nu$, i.e.
$$F_\mu(t) := \int_\RR e^{-ixt} \mu(dx)\quad\text{and}\quad F_\nu(t) := \int_\RR e^{-ixt} \nu(dx),\quad t\in\RR.$$ 
The function $(F_\mu-F_\nu)\in\C^\infty(\RR)$ belongs to $C\{\sqrt{m_{2n}}\}$. In fact, since
$$\frac{d^n}{dt^n}F_\mu(t) = \int_\RR (-ix)^n e^{-ixt} \mu(dx)\quad\text{and}\quad
\frac{d^n}{dt^n}F_\nu(t) = \int_\RR (-ix)^n e^{-ixt} \nu(dx),$$
we get
$$\left|\frac{d^n}{dt^n}(F_\mu-F_\nu)(t)\right|\leq\int_\RR |x|^n \mu(dx)+\int_\RR |x|^n \nu(dx)\leq(c_{\mu}+c_{\nu})\sqrt{m_{2n}},$$
where $c_{\mu}:=\sqrt{\mu(\RR)}$, $c_{\nu}:=\sqrt{\nu(\RR)}$. Moreover, since $\mu$ and $\nu$ have the same moments, we easily get that $\frac{d^n}{dt^n}(F_\mu-F_\nu)(0)=0.$ Then the quasi-analyticity of the class $C\{\sqrt{m_{2n}}\}$ implies that the function $F_\mu-F_\nu$ is identically zero on $\RR$. Consequently, by the injectivity of the Fourier-Stieltjes transform, we have that $\mu=\nu$.\\
\end{proof}

Carleman's condition guarantees that the Hamburger moment problem is determinate unless the even moments tend to infinity quite rapidly. However, this criterion has the disadvantage to only give a sufficient condition for the moment problem to be determinate on $\RR$. Indeed, there exist Hamburger moment sequences $(m_n)_{n\in\NN_0}$ such that  $\sum_{n=1}^\infty \frac{1}{\sqrt[2n]{m_{2n}}} < \infty$ but the correspondent moment problem is determinate (see e.g.~\cite{Stoy87} for examples).

When we consider a Stieltjes moment sequence, we need to be careful in distinguishing the determinacy in the sense of Stieltjes from the one in the sense of Hamburger. Obviously, an indeterminate Stieltjes moment problem is also an indeterminate Hamburger moment problem. However, there are determinate Stieltjes moment problems which are indeterminate in the sense of Hamburger. Regarding the determinacy of the Stieltjes moment problem, we have the following sufficient criterion (see ~\cite[Chapter~ VIII]{Carl26}).
 
\begin{theorem}\label{UniqOneDimStieltjes}\ \\
Let $m=(m_n)_{n\in\NN_0}$ be the moment sequence of $\mu\in\mathcal{M}^*(\RR_+)$. If 
\begin{equation}\label{StieltCond}
\sum_{n=1}^\infty \frac{1}{\sqrt[2n]{m_{n}}} = \infty,
 \end{equation}
then $\mu$ is the unique measure in $\mathcal{M}^*(\RR_+)$ realizing $m$.
\end{theorem}
Condition $\eqref{StieltCond}$ is well-know as \emph{Stieltjes' condition} since it is sufficient for the determinacy of the Stieltjes moment problem. 
\proof\ \\
Let us consider the measure $\nu$ defined on $\RR$ as follows 
$$d\nu(x):=\frac 12\left(\Ii_{[0,+\infty)}(x)d\mu(x^2)+\Ii_{(-\infty, 0]}(x)d\mu(x^2)\right).$$
Then we have that $\nu\in\mathcal{M}^*(\RR)$ and its moment sequence $q=(q_n)_{n\in\NN_0}$ is such that
$q_{2n}=m_n$ and $q_{2n+1}=0$, for all $n\in\NN_0$.
The conclusion follows by Theorem \ref{UniqOneDimHamb} applied to the sequence $q$.
\\
\endproof

This demonstrates that in the general $K-$moment problem, the geometry of $K$ deeply influences its determinateness. Another example is when~$K$ is compact. In fact, if two measures $\mu,\nu\in\mathcal{M}^*(\RR)$ have both compact support~$K$ and the same moment sequence $m$, then by the Stone-Weirstrass theorem we directly get $\mu=\nu$. However, if only one of the two measures has compact support~$K$, then we can still conclude that the correspondent $K-$moment problem for $m$ is determinate, using Carleman's theorem \ref{UniqOneDimHamb} and the following inequality  
$$
m_{2n}=\int_K x^{2n}\mu(dx)\leq\mu(K) \max_{x\in K}x^{2n} ,\quad\forall n\in\NN_0.
$$
The impact of the geometry of the support on the uniqueness of the realizing measure has been extensively treated in~\cite{PutScheid06}, where the authors proved that if $K$ is one-dimensional and virtually compact then every $K-$moment problem is determinate (see~\cite{Scheid03} and~\cite[Remark 3.4]{PutScheid06} for the notion of virtually compact set and recall that such a set is not necessarily compact~\cite[Example 6.3]{PutScheid06}). On the other hand, they showed that there exists a large class of non-virtually compact sets of dimension one which support indeterminate moment sequences. However, as far as we know, it is still open the question if for any $K$ not virtually compact it is possible to construct an indeterminate $K-$moment problem.\\

The quasi-analyticity also plays a fundamental role in the analysis of the moment problem from an operator theoretical point of view. The uniqueness results given in this section for the one-dimensional moment problem follow indeed from the quasi-analytic vectors theorem. A comprehensive exposition about the classical Hamburger and Stieltjes moment problems via methods from the self-adjoint extension theory of symmetric operators is given by Simon in~\cite{Sim98}. In this paper, we will describe the operator theoretical approach only for the higher-dimensional moment problem, because for $d\geq 2$ the quasi-analyticity is already essential to prove the existence of a solution and not only its uniqueness.

\subsection{Determinacy conditions in the multidimensional case}\label{Subs-dDimMP}
The determinacy of the $d-$dimensional moment problem for $d\geq 2$ is a more delicate question, but thanks to quasi-analyticity it is possible to get interesting results also in this case. For a detailed review about this topic, see the comprehensive work of Putinar and Schm\"udgen~\cite{PutSchm08}. The quasi-analyticity of functions in several variables has been already treated in~\cite{BochTayl39, Boch50, Ehr70}. However, we introduce here the analogue of Theorem~\ref{DJ-C} for quasi-analytic classes of functions on $\RR^d$, proposing a recent proof due to de Jeu (see ~\cite[Theorem~B.1]{DeJeu04}). 
\begin{theorem}\label{DeJeuThm}\ \\
For $j=1,\ldots, d$, let $(M_j(n))_{n\in\NN_0}$ be a sequence of positive real numbers s.t.
\be\label{MultCarlGen}
\forall\, j\in\{1,\ldots, d\},\quad\sum\limits_{n=1}^\infty \frac{1}{\beta_j(n)}=\infty\,\text{ with }\, \beta_j(n):=\inf\limits_{k\geq n}\sqrt[k]{M_j(k)}.
\ee
Let $f\in\C^\infty(\RR^d)$ and assume that there exist $A,B\geq 0$ such that for any $\alpha=(\alpha_1,\ldots,\alpha_d)\in\NN_0^d$
$$\|D^\alpha f\|_\infty\leq AB^{|\alpha|}\prod_{j=1}^dM_j(\alpha_j),$$
where $D^\alpha f$ denotes the partial derivative $\frac{\partial^{\left|\alpha\right|}}{\partial x_1^{\alpha_1}\cdots\partial x_d^{\alpha_d}}f$, $\left|\alpha\right|:=\sum\limits_{i=1}^d\alpha_i$ and $\left\|(D^\alpha f)\right\|_\infty:=\sup\limits_{\x\in\RR^d}\left|D^\alpha f(\x)\right|$. If $(D^\alpha f)(0)=0$, $\forall\alpha\in\NN_0^d$, then $f\equiv 0$ on~$\RR^d$.
\end{theorem} 
\begin{remark}\ \\
Note that \eqref{MultCarlGen} is equivalent to require that, for each fixed $j\in\{1,\ldots,d\}$, any of the conditions (a),(c),(d),(e) in Theorem~\ref{DJ-C} is fulfilled by $(M_j(n))_{n\in\NN_0}$.
\end{remark}
\proof\ \\
For $d=1$, the result reduces to Theorem~\ref{DJ-C}. Assume that Theorem~\ref{DeJeuThm} holds for the dimension $d-1$.  For any $\alpha_1,\ldots,\alpha_{d-1}\in\NN_0$, let
$\phi_{\alpha_1,\ldots,\alpha_{d-1}}:\RR\to\RR$ be defined by
$$\phi_{\alpha_1,\ldots,\alpha_{d-1}}(b):=\frac{\partial^{\alpha_1+\cdots+\alpha_{d-1}}}{\partial x_1^{\alpha_1}\cdots\partial x_{d-1}^{\alpha_{d-1}}}f(0,\ldots, 0, b),\quad \forall b\in\RR.$$
Then, all the derivatives of $\phi_{\alpha_1,\ldots,\alpha_{d-1}}$ vanish at $0\in\RR$ by assumption. Moreover, for any $\alpha_d\in\NN_0$ and any $b\in\RR$,
\begin{equation*}
\left|\frac{{d}^{\alpha_d}}{dx_d^{\alpha_d}}\phi_{\alpha_1,\ldots,\alpha_{d-1}}(b)\right|
\leq\left(A\,B^{(\alpha_1+\cdots+\alpha_{d-1})}\prod_{j=1}^{d-1}M_j(\alpha_j)\right)B^{\alpha_d}M_d(\alpha_d).
\end{equation*}
Then by Theorem~\ref{DJ-C}, we have that $\phi_{\alpha_1,\ldots,\alpha_{d-1}}$ is identically zero on $\RR$, for arbitrary $\alpha_1,\ldots,\alpha_{d-1}\in\NN_0$. For each $b\in\RR$, define the function $\psi_b:\RR^{d-1}\to\RR$ as 
$\psi_b(x_1,\ldots, x_{d-1}):=f(x_1,\ldots, x_{d-1},b),$ for any $(x_1,\ldots,x_{d-1})\in\RR^{d-1}$. The previous argument shows that $\psi_b$ fulfills all the assumptions of Theorem~\ref{DeJeuThm} for $d-1$. By inductive assumption, for all $b\in\RR$ we have therefore that $\psi_b$ is identically zero on~$\RR^{d-1}$. Hence, $f$ is identically zero on $\RR^d$.
\endproof

Let us come back to the determinacy question for the higher-dimensional version of the classical Hamburger moment problem. Namely, we ask whether a measure $\mu\in\mathcal{M}^*(\RR^d)$, with $d\geq 2$, is uniquely determined by its moments without any restriction on its support. A fundamental sufficient criterion for uniqueness in this case was obtained by Petersen in \cite{Pet82}.  
\begin{theorem}\label{PetersenThm}\ \\
Let $\mu\in\mathcal{M}^*(\RR^d)$ with $d\geq 2$. For $j=1,\ldots, d$, let $\pi_j:\RR^d\to\RR$ be given by $\pi_j(x_1,\ldots, x_d):=x_j$ and denote by $\mu_{\pi_j}$ the $j-$th marginal measure of $\mu$, i.e. the image measure of $\mu$ under the mapping $\pi_j$. If all the marginal measures $\mu_{\pi_1},\ldots, \mu_{\pi_d}$ are determinate, then $\mu$ is determinate.
\end{theorem}
Petersen proved this result by density arguments on polynomials and he also showed that the converse is not true (see~\cite{Pet82} for a simple example of determinate measure for which not all marginal measures are determinate). Using Theorems~\ref{PetersenThm} and~\ref{UniqOneDimHamb}, we easily get the following.
\begin{theorem}\label{UniquenessHambMultiDim}\ \\
Let $\mu,\nu\in\mathcal{M}^*(\RR^d)$ have the same moment sequence $m=(m_{\alpha})_{\alpha\in\NN_0^d}$. If 
\be\label{MultCarlCondFunct}
  \sum\limits_{n=1}^{\infty}{L_m(x_j^{2n})}^{-\frac{1}{2n}}=\infty,\quad \forall\, j=1,\dots,d,
  \ee
then $\mu=\nu$.
\end{theorem}
An alternative proof of Theorem \ref{UniquenessHambMultiDim} has been recently provided by de~Jeu in~\cite[Theorem~2.3]{DeJeu03}, using Theorem \ref{DeJeuThm} and the observation that \eqref{MultCarlCondFunct} implies \eqref{MultCarlGen} for the sequence $(\sqrt{M_j(2n)})_{n\in\NN_0}$ given by $M_j(h):=L_m(x_j^{h})$ for any $h\in\NN_0$. The proof by de Jeu is very close to the one of Theorem~\ref{UniqOneDimHamb}.

Condition \eqref{MultCarlCondFunct} is well-known as \emph{multivariate Carleman's condition} and it is a sharp determinacy condition for the multivariate moment problem in the following sense.
\begin{theorem}\label{CarlSharp}\ \\ 
Let $(M_n)_{n\in\NN_0}$ be a log-convex sequence of positive real numbers with $M_0=1$. Then the following are equivalent.
\begin{enumerate}
\item\label{uno} The class $C\{M_n\}$ is quasi-analytic.
\item\label{due} For any $\mu,\nu\in\mathcal{M}^*(\RR^d)$ having the same moment multisequence and such that there exists a positive constant $c$ for which
$$\max\left(\int_{\RR^d}{\|\x\|^{2n}\,\mu(d\x)},\int_{\RR^d}{\|\x\|^{2n}\,\nu(d\x)}\right)\leq c\, M_{2n},\quad\forall\,n\in\NN,$$ we have that $\mu=\nu$. (Note that $\|\cdot\|$ denotes the Euclidean norm on $\RR^d$.)
\end{enumerate} 
\end{theorem}
From Theorem~\ref{DJ-C}, Theorem~\ref{UniquenessHambMultiDim} and Lemma~\ref{CarlemanPari}, it easily follows that (\ref{uno}) implies~(\ref{due}) in Theorem~\ref{CarlSharp}. The converse is instead due to Belisl\'e et al. in \cite{BeMaRa97} and we sketch here the main scheme of their proof for $d=1$.

\proof{\it of (\ref{due})$\Rightarrow$(\ref{uno}) in Theorem~\ref{CarlSharp}}\ \\
Suppose that $C\{M_n\}$ is not quasi-analytic and let us take $\varphi\in C\{M_n\}$ as given by Lemma \ref{lemmino}. W.l.o.g. we can assume that the support of $\varphi$ is contained in $[a,b]$ with $0<a <b$.  
For any $A\subseteq\RR$, let us define 
\[
\omega(A):= \int_A Re (\mathcal{F}(\varphi)^2(x)) dx
\]
where $\mathcal{F}(\varphi)$ denotes the Fourier transform of $\varphi$. Then it is easy to show that for any $n \in \mathbb{N}_0$,
$$
\int x^n  d\omega(x) = D^n \mathcal{F}{\omega} (0) =0
\quad\text{and}\quad
 \int x^{2n }  |\omega|(dx)\leq \|\varphi\|_{L^1}M_{2n}.$$
By taking $\mu:=\omega^+$, $\nu:=\omega^-$ and $c:=\|\varphi\|_{L^1}$, the previous relations respectively give that $\mu$ and $\nu$ have the same moments and the following holds
\[
\max\left(\int_{\RR}{x^{2n}\,\mu(dx)},\int_{\RR^d}{x^{2n}\,\nu(dx)}\right) \leq \int x^{2n} d |\omega |(x)\leq c\,M_{2n}.
\]
\endproof
After Petersen, many other sufficient criteria for the multivariate determinacy were developed using polynomial and rational approximation (see e.g.\!~\cite{Fug83, Ber95, PutScheid06, PutSchm08}). All these results use that partial determinacy guarantees the uniqueness of the solution of the multidimensional Hamburger moment problem. However, partial determinacy can be used to prove also the existence part of the moment problem. The first results in this direction were proved by Shohat and Tamarkin in~\cite{Sh-Tam43}, by Devinatz in~\cite{Dev57} and by \`Eskin in~\cite{Esk60}. In these works the authors showed how the determinacy of certain 1-sequences derived from a semidefinite $d-$sequence $m$ ensures both the existence and the uniqueness of a realizing measure for $m$. Nussbaum in~\cite{Nuss65} not only reproved these results with different methods, but also gave the following stronger theorem, which we present here in the form given by Berg in~\cite{Berg87}.
\begin{theorem}\label{HambMultiDim}\ \\
Let $d\geq 2$ and let $m=(m_{\alpha})_{\alpha\in\NN_0^d}$ be a positive semidefinite multisequence fulfilling the multivariate Carleman condition \eqref{MultCarlCondFunct}, then there exists a unique non-negative Borel measure $\mu\in\mathcal{M}^*(\RR^d)$ realizing $m$. 
\end{theorem}
The proof of this result uses the theory of self-adjointness extensions and makes clear that the multivariate Carleman condition is essential not only for the determinacy but also for the existence of the realizing measure. In fact, as we already mentioned above, the condition of positive semidefiniteness of $m$ solely does not imply the existence of a realizing measure on $\RR^d$ when $d\geq 2$ (see Example~6 in \cite{Schm03}). In other words, we will see that we cannot prove an equivalent of Hamburger's existence theorem for higher dimensions without assuming a further condition which guarantees that certain finitely many symmetric self-adjoint operators pairwise strongly commute. 

Before proving Theorem~\ref{HambMultiDim}, let us recall some preliminary notions and results from spectral theory. In the following, for an unbounded operator $T$ on a Hilbert space $\mathcal{H}$, we will denote by $\mathcal{D}(T)$ its domain, which we will suppose to be a dense linear subspace of $\mathcal{H}$. For the classical definitions of symmetric, self-adjoint and essentially self-adjoint operators see for example~\cite[Vol.~I, Chapter~VIII]{ReSi75}. The main tool used by Nussbaum in his proof is the concept of quasi-analytic vector that is intimately related, as we will see, to the multivariate Carleman condition and so to the quasi-analyticity of functions on $\RR^d$.
From now on we denote by $\mathcal{D}^{\infty}(T):=\bigcap\limits_{n=1}^{\infty}{\mathcal{D}(T^n)}$ and by $\mathcal{D}^{qa}(T)$ the set of all quasi-analytic vectors for $T$, i.e.  all vectors $v\in \mathcal{D}^{\infty}(T)$ such that
$\sum\limits_{n=1}^{\infty}{||T^nv||^{-\frac{1}{n}}}=\infty.
$

The motivation of Nussbaum in \cite{Nuss65} was to generalize the classical analytic vectors theorem due to Nelson (see~\cite{Nels59}) to the setting of quasi-analytic vectors. Indeed, he managed to prove this result reducing the situation to Theorem~\ref{UniqOneDimHamb}. For convenience, let us restate here Nussbaum's quasi-analytic vectors theorem (see~\cite[Theorem~2]{Nuss65} and~\cite[Theorem~7.14]{Schmu11}).
\begin{theorem}\label{NussbaumForQA}\ \\
Let $T$ be a symmetric operator on a Hilbert space $\mathcal{H}$ and suppose that its domain $\mathcal{D}(T)$ contains a total set $\mathcal{D}$ of quasi-analytic vectors, i.e. $\mathcal{D}\subseteq\mathcal{D}^{qa}(T)$ and $\overline{span(\mathcal{D})}=\mathcal{H}$. Then $T$ is essentially self-adjoint.
\end{theorem}

However, to solve the multidimensional moment problem we need more, namely the strong commutativity of a pair of operators (see ~\cite[Theorem 6]{Nuss65} and ~\cite[Theorem 7.18]{Schmu11}).

\begin{theorem}\label{NussbaumForTWOQA}\ \\
Let $A$ and $B$ be two symmetric operators on a Hilbert space $\mathcal{H}$. Let $\mathcal{D}$ be a set of vectors in $\mathcal{H}$ which are quasi-analytic for both $A$ and $B$ and such that
$A\mathcal{D}\subset\mathcal{D}$, $B\mathcal{D}\subset\mathcal{D}$, $AB\phi=BA\phi$, for all $\phi\in\mathcal{D}$. If $\mathcal{D}$ is total in $\mathcal{H}$, then the closures $\overline{A}$ and~$\overline{B}$ are strongly commuting self-adjoint operators. Namely, for all $s,t\in\RR$, $e^{is\overline{A}}e^{it\overline{B}}=e^{it\overline{A}}e^{is\overline{B}}$.
\end{theorem}

\begin{remark}\label{commutativitysullepotenze}\ \\
Note that the hypotheses $A\mathcal{D}\subset\mathcal{D}$ and $B\mathcal{D}\subset\mathcal{D}$ guarantee that $\mathcal{D}\subset\mathcal{D}(A^nB^m)$ for any $n,m\in\NN_0$. Then it is easy to see, by induction, that the assumption $AB\phi=BA\phi$ for all $\phi\in\mathcal{D}$ implies $$
A^mB^n\phi=B^nA^m\phi,\quad\forall\, m,n\in\NN_0,\,\forall\,\phi\in\mathcal{D}.
$$ However, this is not sufficient to conclude the strong commutativity of $\overline{A}$ and~$\overline{B}$ (c.f.~\cite[Section VIII.5, Example~1]{ReSi75}). 
\end{remark}
\begin{proof}(of Theorem~\ref{NussbaumForTWOQA})\ \\
Since $\mathcal{D}\subseteq\mathcal{D}^{\text{qa}}({A})$, $\mathcal{D}\subseteq\mathcal{D}^{\text{qa}}({B})$ and $\mathcal{D}$ is total in $\mathcal{H}$, by Theorem~\ref{NussbaumForQA}, the operators $A$ and $B$ are both essentially self-adjoint, i.e. their closures $\overline{A}$ and $\overline{B}$ are self-adjoint. 
In order to show that these operators also strongly commute, we need to use quasi-analyticity of functions in two variables. \\
Given $\phi\in\mathcal{D}$, let us consider the functions
 \begin{eqnarray*} 
 F_1:&\RR^2&\to\mathbb{C}\nonumber\\
&(a,b)& \mapsto \langle e^{ib\overline{B}}\phi,e^{-ia\overline{A}}\phi\rangle
\end{eqnarray*}
and
\begin{eqnarray*} 
 F_2:&\RR^2&\to\mathbb{C}\nonumber\\
&(a,b)& \mapsto \langle e^{ia\overline{A}}\phi,e^{-ib\overline{B}}\phi\rangle\,.
\end{eqnarray*}
It is easy to show that $F_1,F_2\in\mathcal{C}^{\infty}(\RR^2)$. Moreover, for all $\alpha_1,\alpha_2\in\NN_0$ 
\begin{eqnarray*}
	\frac{\partial^{\alpha_2}}{\partial b^{\alpha_2}}\frac{\partial^{\alpha_1}}{\partial a^{\alpha_1}}F_1(a,b)
	&=&{i^{{\alpha_2}+{\alpha_1}}}\langle \overline{B}^{\alpha_2} e^{ib\overline{B}}\phi,\overline{A}^{\alpha_1}e^{-ia\overline{A}}\phi\rangle
\end{eqnarray*}
and 
	\[\frac{\partial^{\alpha_2}}{\partial b^{\alpha_2}}\frac{\partial^{\alpha_1}}{\partial a^{\alpha_1}}F_2(a,b)={i^{{\alpha_2}+{\alpha_1}}}\langle \overline{A}^{\alpha_1}e^{ia\overline{A}}\phi,\overline{B}^{\alpha_2} e^{-ib\overline{B}}\phi\rangle.
\]
\vspace{-0.15cm}\\
Hence, by Remark~\ref{commutativitysullepotenze} we get that
\begin{equation}
\frac{\partial^{\alpha_2}}{\partial b^{\alpha_2}}\frac{\partial^{\alpha_1}}{\partial a^{\alpha_1}}F_1(0,0)=\frac{\partial^{\alpha_2}}{\partial b^{\alpha_2}}\frac{\partial^{\alpha_1}}{\partial a^{\alpha_1}}F_2(0,0).
\label{derF1-F2=0}
\end{equation}
For all $\alpha_1,\alpha_2\in\NN_0$, we also have that 
\be\label{derF1-F2bound}
\left|\frac{\partial^{\alpha_2}}{\partial b^{\alpha_2}}\frac{\partial^{\alpha_1}}{\partial a^{\alpha_1}}\left(F_1-F_2\right)(a,b)\right|\leq2\,M_1({\alpha_1})M_2({\alpha_2}),
\ee
where we set for any $k\in\NN_0$
	\[M_1(k):=||\overline{A}^k\phi||\quad \text{and}\quad M_2(k):=|| \overline{B}^k\phi ||.
\] 
Both $(M_1(k))_{k\in\NN_0}$ and $(M_2(k))_{k\in\NN_0}$ are log-convex because they are defined by norms. The quasi-analyticity of $\phi$ for both $A$ and $B$ implies that
\be\label{qaPhi}
\sum\limits_{k=1}^{\infty}\frac {1}{\sqrt[k]{M_1(k)}}=\infty\quad\text{and}\quad\sum\limits_{k=1}^{\infty}\frac {1}{\sqrt[k]{M_2(k)}}=\infty.\ee
Therefore, by Theorem~\ref{DeJeuThm}, the relations \eqref{derF1-F2=0}, \eqref{derF1-F2bound} and \eqref{qaPhi} imply that the function $F_1-F_2\equiv0$ on~$\RR^2$. Then
	\begin{equation*}
	\langle e^{ib\overline{B}}\phi,e^{-ia\overline{A}}\phi\rangle
	=\langle e^{ia\overline{A}}\phi,e^{-ib\overline{B}}\phi\rangle,\quad\forall a,b\in\RR,\,\forall \phi\in \mathcal{D},
	\end{equation*}
	which also holds for all $\phi\in \mathcal{H}$, since $\mathcal{D}$ is total in $\mathcal{H}$ and the operators $e^{ia\overline{A}}$ and $e^{ia\overline{B}}$ are continuous. 
	Then the conclusion follows by polarization identity.\\
\end{proof}

\begin{proof}(of Theorem~\ref{HambMultiDim} for $d=2$)\ \\
Let $L_m$ be the Riesz functional on $\RR[\textbf{x}]=\RR[x_1,x_2]$ associated to the sequence $m$ (see~Definition~\ref{Lm}). We will apply to this functional the well-known Gelfand-Naimark-Segal (GNS) construction and then we will use the spectral theorem for pairwise strongly commuting self-adjoint extensions of the multiplication operators defined on the Hilbert space given by the GNS-construction.

Since $m$ is a positive semidefinite sequence, the bilinear form given by
$\langle f, g \rangle:=L_m(fg)$
is a quasi-inner product on $\RR[\x]$ and by the Cauchy-Schwarz inequality it follows that the subset $N:=\{h\in\RR[\textbf{x}]: L_m(h^2)=0\}$ is an ideal of the algebra of polynomials $\RR[\x]$. Let $\mathcal{H}_m$ be the completion of the pre-Hilbert space $\RR[\x]/N$ equipped with the inner product $\langle\cdot,\cdot\rangle$. For $j=1,2$, we introduce the operator
$X_j:\RR[\textbf{x}]/N\to\RR[\textbf{x}]/N$ defined by $$X_j\big(h(x_1,x_2)\big):= x_j\, h(x_1,x_2),\,\,\text{ for any }\,h\in\RR[\textbf{x}]/N.$$ Then $X_1$, $X_2$ and $\mathcal{D}:=\{x_1^sx_2^n| s,n\in\NN_0\}$ fulfill all the assumptions of Theorem~\ref{NussbaumForTWOQA}. We only show that $\mathcal{D}$ is a set of quasi-analytic vectors for both ${X_1}$ and ${X_2}$. Let us fix $s,n\in\NN_0$, then by Cauchy-Schwarz's inequality we get that for any $k\in\NN$
\be\label{CSbound}
||{X_1}^{k}x_1^sx_2^n||^2\leq\left(L_m(x_1^{4k+4s})\right)^{\frac 12}\left(L_m(x_2^{4n})\right)^{\frac 12}.
\ee
Now, let us define the sequence $M_j(k):=L_m(x_j^k)$ for $j=1,2$. The log-convexity of the sequence $(M_1(k))_{k\in\NN_0}$ easily follows by the Cauchy-Schwarz inequality for the inner product $\langle \cdot, \cdot\rangle$. By Theorem \ref{DJ-C} and Lemma~\ref{CarlemanPari}, the multivariate Carleman condition \eqref{MultCarlCondFunct} for $j=1$ guarantees that $C\{M_1(k)\}$ is quasi-analytic. By Lemma~\ref{LemmadelloShift}, we get that for the fixed $s\in\NN_0$ the class $C\{M_1(k+s)\}$ is also quasi-analytic. Then, by Lemma~\ref{CarlemanPari}, $C\{\sqrt[4]{M_1(4k+4s)}\}$ is quasi-analytic. Since $M_2(4n)$ is constant in $k$, Proposition~\ref{ConstTimesCarl} guarantees that $ \sum\limits_{k=1}^{\infty}\frac{1}{\sqrt[4k]{M_1(4k+4s)M_2(4n)}}=\infty$. This together with \eqref{CSbound} implies that
$\sum\limits_{k=1}^{\infty}{||{X_1}^{k}x_1^sx_2^n||^{-\frac{1}{k}}}=\infty,$ i.e. $x_1^sx_2^n$ is a quasi-analytic vector for $X_1$. The same proof applies to~$X_2$.

Theorem~\ref{NussbaumForTWOQA} guarantees that the closures $\overline{X_1}$ and $\overline{X_2}$ of $X_1$ and~$X_2$, respectively, are strongly commuting self-adjoint operators. By applying the spectral theorem to $\overline{X_1}$ and $\overline{X_2}$, we get that there exists a unique non-negative measure $\mu\in\mathcal{M}^*(\RR^2)$
such that for any $\alpha_1,\alpha_2\in\NN_0$
\be\label{equaz1}
	\int_{\RR^2}{x_1^{\alpha_1}x_2^{\alpha_2}\mu(dx_1,dx_2)}=\langle1,\underbrace{\overline{X_1}\cdots\overline{X_1}}_{\alpha_1\ \text{times}}\underbrace{\overline{X_2}\cdots\overline{X_2}}_{\alpha_2\ \text{times}}\cdot 1\rangle.\ee
On the other hand, we have that for any $\alpha_1,\alpha_2\in\NN_0$
\be\label{equaz2}\langle1,\underbrace{\overline{X_1}\cdots\overline{X_1}}_{\alpha_1\ \text{times}}\underbrace{\overline{X_2}\cdots\overline{X_2}}_{\alpha_2\ \text{times}}\cdot 1\rangle=\langle1,(X_1^{\alpha_1}X_2^{\alpha_2})(1)\rangle=L_m(x_1^{\alpha_1}x_2^{\alpha_2})=m_{(\alpha_1,\alpha_2)}.\ee
\vspace{-0.1cm}\\
By \eqref{equaz1} and \eqref{equaz2}, we conclude that $ \int_{\RR^2}{\x^{\alpha}\mu(d\x)}=m_{\alpha}$ for any $\alpha\in\NN_0^2$, i.e.\! the sequence $m$ is realized on $\RR^2$ by the measure $\mu$. Moreover, since $m$ fulfills~\eqref{MultCarlCondFunct} by assumption, Theorem~\ref{UniquenessHambMultiDim} for $n=2$ guarantees that $\mu$ is the unique measure realizing $m$ on $\RR^2$.\\
\end{proof}

Concerning the Stieltjes moment problem in higher dimensions, it is possible to obtain sufficient determinacy conditions using the quasi-analyticity of the Fourier-Laplace transform of a measure supported on $\RR^d_+$. This technique is used in ~\cite[Section 2.4]{PutSchm08}, where the authors proved different determinacy conditions corresponding to the different quasi-analyticity criteria given in ~\cite{BochTayl39} and ~\cite{Boch50}. Following the proof in the one-dimensional case (see Theorem~\ref{UniqOneDimStieltjes}), it is possible to derive from Theorem \ref{UniquenessHambMultiDim} the following sufficient condition for the determinacy of the multidimensional Stieltjes moment problem (see ~\cite{DeJeu03} for a detailed proof of this result).
\begin{theorem}\ \\
Let $m=(m_\alpha)_{\alpha\in\NN_0^d}$ be the moment sequence of a measure $\mu\in\mathcal{M}^*(\RR^d_+)$. If 
$$
\sum\limits_{k=1}^{\infty}{L_m(x_j^{k})}^{-\frac{1}{2k}}=\infty,\quad \forall\, j=1,\dots,d,
$$
then $\mu$ is the unique measure realizing $m$ on $\RR^d_+$.
\end{theorem}

As well as in the one-dimensional case, the geometry of the support of a measure on $\RR^d$ with $d\geq 2$ can be used to derive other determinacy conditions. First of all, the compactness of $K\subset \RR^d$ guarantees the determinacy of the multivariate $K-$moment problem for any $d\in\NN$ (the considerations made at the end of Subsection~\ref{Subs-1DimMP} can be straightforwardly generalized to higher dimensions). Moreover, in~\cite[Section 3]{PutScheid06} the authors showed higher-dimensional determinacy criteria based on the geometry of the support and provided examples of non-compact higher-dimensional sets which support determinate measures (see also~\cite[Section 9]{PutSchm08} for a summary of the results in~\cite{PutScheid06}). Another powerful method to study the determinacy of the multidimensional moment problem is to use disintegration techniques. In particular, Putinar and Schm\"udgen have recently proved through such techniques a general result which reduces the determinacy question to lower dimensions and it has a very broad class of applications,~\cite[Section 8]{PutSchm08}.

\section{Uniqueness in the infinite-dimensional moment problem}\label{InftyDimMP}
\subsection{The moment problem on conuclear spaces}\label{Sec-Prel}
In the following we are going to introduce an infinite-dimensional version of the moment problem, in particular we will consider the moment problem on conuclear spaces. For simplicity, from now on, all the spaces are assumed to be separable and real.

Let us consider a family $(H_k)_{k\in K}$ of Hilbert spaces ($K$ is an index set containing~$0$) which is directed by topological embedding, i.e. 
$$\forall\ k_1,k_2\in K\,\,\exists\, k_3\,:\, H_{k_3}\subseteq H_{k_1}\, ,\, H_{k_3}\subseteq H_{k_2}.$$
We assume that each $H_k$ is topologically embedded into $H_0$.
Let $\Omega$ be the projective limit of the family $(H_k)_{k\in K}$ endowed with the associated projective limit topology and let us assume that $\Omega$ is nuclear, i.e.\! for each $k_1\in K$ there exists $k_2\in K$ such that the embedding $H_{k_2}\subseteq H_{k_1}$ is quasi-nuclear.

Let us denote by $\Omega'$ the topological dual space of $\Omega$. We control the classical rigging by identifying $H_0$ and its dual $H'_0$. With this identification one can define the duality pairing between elements in $H_k$ and in its dual $H'_k=H_{-k}$ using the inner product in $H_0$. For this reason, in the following we will denote by $\langle f, \eta\rangle$ the duality pairing between $\eta\in\Omega'$ and $f\in\Omega$ (see~\cite{B86, BeKo88} for more details).

Consider the $n-$th ($n\in\NN_0$) symmetric tensor power $\Omega^{\otimes n}$ of the space~$\Omega$ which is defined as the projective limit of all $H_k^{\otimes n}$; for $n=0$, $H_k^{\otimes 0}=\RR$. Then its dual space is
\begin{equation}\label{tensorDual}
\left(\Omega^{\otimes n}\right)'=\bigcup_{k\in K}\left(H_k^{\otimes n}\right)'=\bigcup_{k\in K}(H'_k)^{\otimes n}=\bigcup_{k\in K}H_{-k}^{\otimes n},
\end{equation}
which we can equip with the weak topology.

A generalized process $\mu$ is a finite measure defined on the Borel $\sigma-$algebra on~$\Omega'$. Moreover, we say that a generalized process $\mu$ is \emph{concentrated on} a measurable subset $\mathcal{S}\subseteq\Omega'$ if $\mu\left(\Omega'\setminus\mathcal{S}\right)=0$. 

\begin{definition}[Finite $n-$th local moment]\ \\
Given $n\in\NN$, a generalized process $\mu$ on $\Omega'$ has \emph{finite $n-$th local moment} (or local moment of order $n$) if for every $f\in\Omega$ we have
$$\int_{\Omega'}|\langle f, \eta\rangle|^n \mu(d\eta)<\infty.$$
\end{definition}

\begin{definition}[$n-$th generalized moment function]\ \\
Given $n\in\NN$, a generalized process $\mu$ on $\Omega'$ has $n-$th generalized moment function in the sense of $\Omega'$ if $\mu$ has finite $n-$th local moment and if the functional $f\mapsto\int_{\Omega'}|\langle f, \eta\rangle|^n \mu(d\eta)$ is continuous on $\Omega$. In fact, by the Kernel Theorem, for such a generalized process $\mu$ there exists a symmetric functional $m^{(n)}_{\mu}\in(\Omega^{\otimes n})'$, which will be called the \emph{$n-$th generalized moment function in the sense of $\Omega'$}, such that
for any $f^{(n)}\in\Omega^{\otimes n}$ we have 
\begin{equation*}\label{nMomTensor}
 \langle f^{(n)}, m_{\mu}^{(n)} \rangle =\int_{\Omega'} \langle f^{(n)}, \eta^{\otimes n} \rangle \mu(d\eta).
 \end{equation*} 
By convention, $m_{\mu}^{(0)}:=\mu(\Omega')$.
\end{definition}

In analogy to the finite-dimensional case, we will denote by
$\mathcal{M}^*(\mathcal{S})$ the collection of all generalized processes concentrated on a measurable subset $\mathcal{S}$~of~$\Omega'$ with generalized moment functions (in the sense of $\Omega'$) of any order. Moreover, let us simply denote by $\mathcal{F}(\Omega')$ the collection of all infinite sequences $(m^{(n)})_{n\in\NN_0}$ such that each $m^{(n)}\in\left(\Omega^{\otimes n}\right)'$ is a symmetric functional, namely the tensor product $\left(\Omega'\right)^{\otimes n}$ is considered to be symmetric.

The full moment problem, which in this infinite-dimensional context is often called the \emph{full realizability problem}, addresses exactly the following question.
\begin{probl}[Full realizability problem on $\mathcal{S}\subseteq\Omega'$]\label{RealProb}\  \\
Let $\mathcal{S}$ be a measurable subset of $\Omega'$ and let $m=(m^{(n)})_{n\in\NN_0}\in\mathcal{F}(\Omega')$. Find a generalized process $\mu\in\mathcal{M}^*(\mathcal{S})$ such that $m^{(n)}=m^{(n)}_\mu$ for all $n\in\NN_0$, i.e. $m^{(n)}$ is the $n-$th generalized moment function of $\mu$ for any $n\in\NN_0$.
\end{probl}
If such a measure $\mu$ does exist we say that $m$ is \emph{realized} by $\mu$ on~$\mathcal{S}$. Note that the statement of the problem requires that one finds a measure concentrated on~$\mathcal{S}$ and not only on $\Omega'$. \

An obvious positivity property which is necessary for an element in $\mathcal{F}(\Omega')$ to be the moment sequence of some measure on $\Omega'$ is the following.
\begin{definition}[Positive semidefinite sequence]\label{PosSemiDef}\ \\
A sequence $m\in\mathcal{F}(\Omega')$ is said to be \emph{positive semidefinite} if for any $f^{(j)}\in\Omega^{\otimes j}$
$$\sum_{j,l=0}^\infty \langle  f^{(j)} \otimes f^{(l)}, m^{(j+l)}\rangle\geq 0.$$
\end{definition}
This is a straightforward generalization of the classical notion of positive semidefiniteness considered in the finite-dimensional moment problem (see Definition~\ref{PSD}). Note that, as we work with real spaces, we choose the involution on $\Omega$ considered in~\cite{BeKo88} to be the identity. 

A measure $\mu\in\mathcal{M}^*(\mathcal{S})$ is called \emph{determinate} on $\mathcal{S}$ if any other $\nu\in\mathcal{M}^*(\mathcal{S})$ having the same generalized moment functions as $\mu$ is equal to~$\mu$.

\subsection{Determinacy condition for the realizability problem on conuclear spaces}
As well as for the $d-$dimensional moment problem with $d\geq 2$, the role of quasi-analyticity in the infinite-dimensional moment problem is fundamental not only to develop sufficient determinacy conditions but also to obtain the existence of a solution. The following notion is the crucial element to get analogues of Theorem~\ref{UniquenessHambMultiDim} and Theorem~\ref{HambMultiDim} for the realizability problem. 
 \begin{definition}[Determining sequence]\label{DefSeq}\ \\
Let $m\in\mathcal{F}(\Omega')$ and let $E$ be a countable total subset of $\Omega$, i.e.\! the linear span of $E$ is dense in~$\Omega$. Let us define the sequence $(m_n)_{n\in\NN_0}$ as follows
\begin{equation}\label{defCond}
m_0:=\sqrt{|m^{(0)}|}\,\text{ and }\,m_n:= \sqrt{\sup_{f_1,\ldots,f_{2n}\in E}|\langle f_1\otimes\cdots\otimes f_{2n},m^{(2n)}\rangle|},\, \forall\,n\geq 1.
\end{equation}
The sequence $m$ is said to be \emph{determining} if and only if there exists a countable total subset $E$ of $\Omega$ such that for any $n\in\NN_0$, $m_n<\infty$ and the class $C\{m_n\}$ is quasi-analytic (see\!~Definition~\ref{qaClass} and Theorem~\ref{DJ-C}).
\end{definition}
Note that from \eqref{tensorDual} it follows that for any sequence $m\in\mathcal{F}(\Omega')$ there exists a sequence $(k^{(n)})_{n\in\NN_0}\subset K$ s.t. for any $n\in\NN_0$ we have $m^{(n)}\in H_{-k^{(n)}}^{\otimes n}$. If we denote by 
\be\label{dEYuri}
d(k^{(n)}, E):=\sup_{f\in E}\|f\|_{H_{k^{(n)}}},\ee then for the $m_n$'s defined in \eqref{defCond} we have
$$
m_n\leq (d(k^{(2n)}, E))^{n}\|m^{(2n)}\|_{H_{-k^{(2n)}}^{\otimes{2n}}}^{\frac 12}.
$$
Hence, we can see that a preferable choice for $E$ is the one for which the sequence $\left(d(k^{(2n)},E)\right)_{n\in\NN}$ grows as little as possible. For instance, in~\cite[Lemma~4.5]{Inf-Ku-Ro} we proved that it is possible to explicitly construct such a set~$E$ in the case when $\Omega$ is the space of all infinitely differentiable functions with compact support in $\RR^d$. This explicit construction is based on quasi-analyticity theory and uses a technique similar to the one of ~\cite[Chapter~4, Section~9]{Gel-Shi68}.\\

Let us prove now the correspondent of Theorem~\ref{UniquenessHambMultiDim} for Problem~\ref{RealProb} in the case $\mathcal{S}=\Omega'$ (c.f.~\cite[Vol.~II, Theorem~2.1]{BeKo88} and \cite{BS71}). 
Before stating the theorem, we need some preliminary considerations.

Let $m\in\mathcal{F}(\Omega')$ be the moment sequence of a measure $\mu$ on $\Omega'$. For any $\varphi\in\Omega$ and any $n\in\NN_0$, we define $$m_{\varphi,n}:=\langle \varphi^{\otimes n}, m^{(n)}\rangle.$$ Therefore, we have $$m_{\varphi,n}=\int_{\Omega'}\langle \varphi, \eta\rangle^n\mu(d\eta)=\int_\RR t^n\mu_{\pi_\varphi}(d t),$$
where $\pi_\varphi(\eta):=\langle \varphi, \eta\rangle$ for all $\eta\in\Omega'$ and $\mu_{\pi_\varphi}$ is the image measure of $\mu$ under $\pi_\varphi$. Note that the sequence $(m_{\varphi,n})_{n\in\NN_0}$ is a log-convex sequence of real numbers and if $\mu$ is a probability then $m_{\varphi,0}=1$.

\begin{theorem}\label{KondrThmUniqueness}\ \\
Let $\Omega'$ be a Suslin space and let $\mu,\nu\in\mathcal{M}^*(\Omega')$ have the same generalized moment sequence $m=(m^{(n)})_{n\in\NN_0}$. If there exists a countable total subset $E$ of $\Omega$ such that for all $\varphi\in E$ the class $C\{m_{\varphi,n}\}$ is quasi-analytic, then $\mu=\nu$. In particular, if $m$ is determining, then the conclusion holds.
\end{theorem}

\proof\ \\
Since for any $\varphi\in E$ the class $C\{m_{\varphi,n}\}$ is quasi-analytic and the sequence $(m_{\varphi,n})_{n\in\NN_0}$ is log-convex, by Lemma~\ref{CarlemanPari} and Theorem~\ref{UniqOneDimHamb}, it follows that $\mu_{\pi_\varphi}=\nu_{\pi_\varphi}$ on $\RR$. To show that $\mu=\nu$ on $\Omega'$, it is enough to prove that $\mu$ and~$\nu$ coincide on all the cylindrical sets
$$C(f_1,\ldots, f_n; B):=\{\eta\in\Omega': (\langle f_1, \eta\rangle,\ldots,\langle f_n, \eta\rangle)\in B\},$$
with $n\in\NN$, $f_1,\ldots, f_n\in E$ and $B$ in the Borel $\sigma-$algebra $\mathcal{B}(\RR^n)$ on $\RR^n$.
In fact, since $E$ is total in $\Omega$ and $\Omega'$ is Suslin, a theorem due to Fernique (see~\cite[Lemma 18]{Sch73}) guarantees that the Borel $\sigma-$algebra on $\Omega'$ is generated by all the cylinders above. 

Since for any $n\in\NN$ and for any $f_1,\ldots, f_n\in E$, we have already proved that $\mu_{\pi_{f_j}}=\nu_{\pi_{f_j}}$ on $\RR$ for all $j\in\{1,\ldots,n\}$, Petersen's Theorem~\ref{PetersenThm} implies that $\mu(C(f_1,\ldots, f_n; B))=\nu(C(f_1,\ldots, f_n; B))$ for any $B\in\mathcal{B}(\RR^n)$. 

In particular, if $m$ is determining then there exists a countable total subset $E$ of $\Omega$ such that $C\{m_n\}$ is quasi-analytic, where $m_n$ is defined as in~\eqref{defCond}. Then for any $\varphi\in E$ and $n\in\NN_0$ we get 
$\sqrt{m_{\varphi, 2n}}\leq m_n,$
which implies that $C\{\sqrt{m_{\varphi, 2n}}\}$ is quasi-analytic. Hence, by Lemma~\ref{CarlemanPari}, $C\{m_{\varphi,n}\}$ is quasi-analytic. Then the conclusion follows by the first part of this proof.\\
\endproof

Let us state now the analogue of Theorem~\ref{HambMultiDim} for Problem~\ref{RealProb} in the case $\mathcal{S}=\Omega'$ (see\!~\cite[Vol.~II, Theorem~2.1]{BeKo88} and ~\cite{BS71}).
\begin{theorem}\label{KondrThm}\ \\
Let $\Omega'$ be a Suslin space. If $m\in\mathcal{F}(\Omega')$ is a positive semidefinite sequence which is also determining, then there exists a unique non-negative generalized process $\mu\in\mathcal{M}^*(\Omega')$ such that for any $n\in\NN_0$ and for any $f^{(n)}\in\Omega^{\otimes n}$
 \begin{equation*}\label{GenRealiz}
 \left\langle f^{(n)}, m^{(n)} \right\rangle=\int_{\Omega'}\left\langle f^{(n)},\eta^{\otimes n}\right\rangle\mu( d\eta).
 \end{equation*}
 \end{theorem}

The original proof of Theorem~\ref{KondrThm} in ~\cite{BeKo88} uses a slightly less general definition of determining sequence. Indeed, the authors require that the class $$C\left\{d(k^{(2n)},E)^n\left\|m^{(2n)}\right\|_{H_{-k^{(2n)}}^{\otimes 2n}}^{1/2}\right\}$$ is quasi-analytic, which in turn implies that $C\{m_n\}$ is also quasi-analytic. Nevertheless, their proof also works using just the bound given by Definition~\ref{DefSeq}. The latter has actually the advantage to guarantee that, whenever $m$ is realizable on $\Omega$, the sequence $(m_n)_{n\in\NN_0}$ is log-convex. This is an essential property to obtain necessary and sufficient conditions for the realizability problem on semi-algebraic sets (see~\cite{Inf-Ku-Ro} for more details on this topic). 

\proof (Sketch)\ \\
The general scheme of the proof of Theorem~\ref{KondrThm} is very similar to the one of Theorem~\ref{HambMultiDim}. As in the finite-dimensional case, the GNS construction is used to define a Hilbert space $\mathcal{H}_m$ associated to the starting positive semidefinite sequence $m\in\mathcal{F}(\Omega')$, which is now a sequence of functionals and no more of real numbers. 
Consider the set $\mathscr{P}_{\Omega}\left(\Omega'\right)$ of all polynomials on $\Omega'$ of the form 
\begin{equation}\label{poly}
P(\eta) := \sum_{j=0}^N\langle f^{(j)},\eta^{\otimes j}\rangle,
\end{equation}
where $f^{(0)}\in\RR$ and $f^{(j)}\in\Omega^{\otimes j}$, $j=1,\ldots,N$ with $N\in\NN$. Let us define
the Riesz functional $L_m$ associated to $m\in\mathcal{F}(\Omega')$ as 
\begin{eqnarray*}
L_m: &\mathscr{P}_{\Omega}\left(\Omega'\right)&\to\RR \\
 \  &P(\eta)=\sum\limits_{n=0}^N\langle p^{(n)} , \eta^{\otimes n}\rangle & \mapsto L_m(P):=\sum_{n=0}^N\langle p^{(n)} , m^{(n)}\rangle .
\end{eqnarray*}
Since $m$ is positive semidefinite, the bilinear form given by
$\langle P, Q \rangle:=L_m(PQ)$
is a quasi-inner product on $\mathscr{P}_{\Omega}\left(\Omega'\right)$. After the factorization of $\mathscr{P}_{\Omega}\left(\Omega'\right)$ w.r.t. $N:=\{P\in\mathscr{P}_{\Omega}\left(\Omega'\right): L_m(P^2)=0\}$ and the subsequent completion of this quotient space, we obtain a Hilbert space $\mathcal{H}_m$. 

For any $e\in E$, let us introduce the multiplication operator
$A_e$ on $\mathcal{H}_m$ defined by $$A_eP:=\sum_{j=0}^N\langle e\otimes f^{(j)},\eta^{\otimes (j+1)}\rangle,\,\text{ for any }P\in\mathscr{P}_{\Omega}\left(\Omega'\right)\,\text{as in \eqref{poly}}.$$ 
Thanks to the determining condition, it is possible to show that the domain of each of these operators contains a countable total subset of quasi-analytic vectors, namely, the set all the polynomials $P_k(\eta):=\langle f_1\otimes\cdots\otimes f_k, \eta^{\otimes k}\rangle$ with $k\in\NN$ and $f_1,\ldots,f_k\in E$. Indeed, for any $e\in E$ and any $k\in\NN$ we have that $P_k\in\cap_{n=1}^\infty \mathcal{D}(A_e^n)$. Moreover, for any $n\in\NN$ we get
\begin{eqnarray*}
\|A_e^n P_k\|&=&\langle A_e^n P_k,A_e^n P_k\rangle^{\frac 12}=L_m(\langle f_1\otimes\cdots\otimes f_k\otimes\underbrace{e\otimes\cdots\otimes e}_{n \text{ times}}, \eta^{\otimes (k+n)}\rangle^2)^{\frac 12}\\
&=&\langle f_1^{\otimes 2}\otimes\cdots \otimes f_k^{\otimes 2}\otimes \underbrace{e\otimes\cdots\otimes e}_{2n \text{ times}}, m^{(2(k+n))}\rangle^{\frac 12}\leq m_{k+n}.
\end{eqnarray*}
By Definition~\ref{DefSeq}, the class $C\{m_n\}$ is quasi-analytic and the sequence $(m_n)_{n\in\NN}$ is log-convex. Then, by Lemma~\ref{LemmadelloShift}, the class  $C\{m_{k+n}\}$ is also quasi-analytic. Hence, the previous estimate shows that each $P_k$ is a quasi-analytic vector for $A_e$. As in the finite-dimensional case, it is possible to show that these operators admit unique self-adjoint extensions, which are pairwise strongly commuting. Therefore, by the spectral theorem for infinitely countable many unbounded self-adjoint operators (see~\cite[Vol.~I, Section~2]{BeKo88}), there exists a unique measure representing those operators. Hence, this spectral measure $\mu$ realizes $m$ on $\RR^E$. 

The final part of the proof consists in showing that the spectral measure is actually supported on $\Omega'$. Moreover, since $m$ is determining by assumption, Theorem~\ref{KondrThmUniqueness} also guarantees that the measure $\mu$ is the unique measure realizing $m$ on $\Omega'$.\\
\endproof

\begin{remark}\ \\
The $d-$dimensional moment problem on $\RR^d$ is a special case of Problem~\ref{RealProb} for $\Omega=H_0=\RR^d$. Hence, an analogue of Theorem~\ref{KondrThm} can be proved also in the finite-dimensional case, where the condition $m:=(m^{(n)})_{n\in\NN_0}\in\mathcal{F}\left(\RR^{d}\right)$ holds for any multisequence of real numbers. In fact, if $\{e_1,\ldots, e_d\}$ denotes the canonical basis of $\RR^d$ then we have that for each $n\in\NN_0$, $$m^{(n)}:=\sum\limits_{\stackrel{n_1,\ldots,n_d\in\NN_0}{n_1+\cdots+n_d=n}}m_{(n_1,\ldots, n_d)}^{(n)} \underbrace{e_1\otimes\cdots\otimes e_1}_{n_1\text{ times}}\otimes\cdots\otimes \underbrace{e_{d}\otimes\cdots\otimes e_d}_{n_d\text{ times}}\in\RR^{dn}.$$ 
The determining condition on $m$ reduces to the requirement that the class 
\begin{equation}\label{classMax}
C\left\{\sqrt{\max\limits_{\stackrel{n_1,\ldots,n_d\in\NN_0}{n_1+\cdots+n_d=2n}}|m_{n_1,\ldots, n_d}^{(2n)}|}\right\}\end{equation}
 is quasi-analytic. This follows by taking $E:=\{e_1,\ldots,e_d\}$ in Definition~\ref{DefSeq}. Note that the quasi-analyticity of the class \eqref{classMax} implies the multivariate Carleman condition~\eqref{MultCarlCondFunct}. Hence, Theorem~\ref{KondrThm} gives here a slightly weaker version than Theorem~\ref{HambMultiDim}, but it is now clear that it is the quasi-analyticity the key stone on which both results are based. Let us also underline that, whenever the starting sequence $m$ is realizable on $\RR^d$, the sequence in \eqref{classMax} is log-convex.
\end{remark}

The infinite-dimensional analogue of the Stieltjes moment problem was considered by \v{S}ifrin in ~\cite{S74}, where he develops the analysis of the infinite-dimensional moment problem on dual cones in conuclear spaces. In particular, applying \v{S}ifrin's result to a generating cone $\mathcal{K}$ of a nuclear space $\Omega$ as before, it is possible to obtain a version of Theorem~\ref{KondrThm} for the realizability problem on the dual cone of such a $\mathcal{K}$, but with the difference that the determining condition is replaced by the requirement that the class $C\{\sqrt{m_n}\}$ is quasi-analytic where the $m_n$'s are defined as in Definition~\ref{DefSeq}. This condition is slightly more general than the one given by \v{S}ifrin in~\cite{S74}, which can be rewritten as  
\be\label{GenStieltCond}
\sum_{n=1}^\infty  \frac{1}{\sqrt[4n]{d(k^{(2n)},E)^{2n}\left\|m^{(2n)}\right\|_{H_{-k^{(2n)}}^{\otimes 2n}}}}=\infty,
\ee 
where $d(k^{(2n)},E)$ is defined as in \eqref{dEYuri} and $E$ is total in $\Omega$. Condition \eqref{GenStieltCond} is called \emph{generalized Stieltjes' condition}. In fact, it is easy to see that the difference between \eqref{GenStieltCond} and the original determining condition given by Berezansky and Kondratiev, i.e.
$$\sum_{n=1}^\infty  \frac{1}{\sqrt[2n]{d(k^{(2n)},E)^{2n}\left\|m^{(2n)}\right\|_{H_{-k^{(2n)}}^{\otimes 2n}}}}=\infty,$$ is the same as the difference between Stieltjes' condition~\eqref{StieltCond} and Carleman's condition~\eqref{CarlCond} in the one-dimensional case. 

As in the finite-dimensional case, the geometry of the support $\mathcal{S}$ again allows to get easier conditions for realizability (see~\cite{Inf-Ku} and \cite{Inf-Ku-Ro} for some examples).

\section{Appendix: Log-convexity and quasi-analyticity}\label{App-LogconvQA}
We conclude this paper with some properties of log-convex sequences which are useful in relation to the quasi-analyticity of the associated classes of functions.
\begin{prop}\label{LogConvexChar} \ \\
For a sequence of positive real numbers $(M_n)_{n\in\NN_0}$  the following are equivalent
\begin{enumerate}[(a).]
\item $(M_n)_{n\in\NN_0}$ is log-convex.
\item $\left(\frac{M_{n}}{M_{n-1}}\right)_{n\in\NN}$ is monotone increasing.
\item $\left(\ln(M_n)\right)_{n\in\NN}$ is convex.
\end{enumerate}
\end{prop}

\proof\ \\
The conditions (a) and (b) are obviously equivalent. If (c) holds, then
\[
2\ln M_{n}
\leq \ln M_{n+1}+ \ln M_{n-1},
\]
which implies (a). Let us assume (b), then for any $n, m, k\in\NN$ such that $n \leq  k\leq  m$ we have
\begin{equation}\label{Con2}
\frac{1}{k-n}\sum_{j=n+1}^k \ln \left( \frac{M_j}{M_{j-1}}\right) \leq
\frac{1}{m-k}\sum_{j=k+1}^m  \ln \left( \frac{M_j}{M_{j-1}}\right),
\end{equation}
where we used the fact that the denominators of the pre-factors are equal to the number of summands in both sums. The inequality \eqref{Con2} is equivalent to
\begin{equation*}
(m-n)\ln{M_k} \leq (k-n) \ln M_m+(m-k)\ln M_n,
\end{equation*}
which gives the convexity of the sequence $\left(\ln(M_n)\right)_{n\in\NN}$ and so (c).\\
\endproof
\begin{corollary}\label{corInc}\ \\
If a sequence of positive real numbers $(M_n)_{n\in\NN_0}$ is log-convex with $M_0=1$, then $(\sqrt[n]{M_n})_{n\in\NN}$ is monotone increasing.
\end{corollary}
\proof\ \\
From (b) in Proposition~\ref{LogConvexChar}, it follows that for any $n\in\NN$
\[
M_n = \frac{M_n}{M_0} = \prod_{j=1}^n \frac{M_j}{M_{j-1}} \leq
\left(\frac{M_n}{M_{n-1}} \right)^n,
\]
which gives
$
M_{n-1}^{n} \leq M_{n}^{n-1},
$
or equivalently,
$
M_{n-1}^{1/n-1} \leq M_{n}^{1/n}.
$\\
\endproof

\begin{lemma}\label{CarlemanPari}\ \\
Assume that $(M_n)_{n\in\NN_0}$ is a log-convex sequence of positive real numbers. $C\{M_n\}$ is quasi-analytic if and only if for some (and hence for any) $j\in\NN$ the class $C\{\sqrt[j]{M_{jn}}\}$ is quasi-analytic.
\end{lemma}
\proof\ \\
W.l.o.g. we can assume that $M_0=1$. (In fact, if $M_0\neq 1$ then one can always apply the following proof to the sequence $(\frac{M_n}{M_0})_{n\in\NN_0}$ by Proposition~\ref{ConstTimesCarl}.)\\
Let us first note that, by Theorem~\ref{DJ-C}, it is enough to prove $\sum\limits_{n=1}^\infty \frac{1}{\sqrt[n]{M_{n}}}=\infty$ if and only if for some $j\in\NN$, $\sum\limits_{n=1}^\infty \frac{1}{\sqrt[jn]{M_{jn}}}=\infty$. By Corollary \ref{corInc}, we have
\begin{eqnarray*}
\sum_{n=1}^\infty \frac{1}{\sqrt[n]{M_{n}}}\!\!\!\!&=&\!\!\!\!\sum_{n=1}^\infty \left(
  \frac{1}{\sqrt[jn]{M_{jn}}} +
  \cdots + \frac{1}{\sqrt[jn+(j-1)]{M_{jn+j-1}}}\right) + \sum_{n=1}^{j-1} 
  \frac{1}{\sqrt[n]{M_{n}}} \\
  &\leq& j \sum_{n=1}^\infty
  \frac{1}{\sqrt[jn]{M_{jn}}} + \sum_{n=1}^{j-1}
  \frac{1}{\sqrt[n]{M_{n}}},
\end{eqnarray*}
which gives the necessity part. On the other hand, if $\sum\limits_{n=1}^\infty \frac{1}{\sqrt[jn]{M_{jn}}}$ diverges for some $j\in\NN$, then also $\sum\limits_{n=1}^\infty \frac{1}{\sqrt[n]{M_{n}}}$ diverges since it contains more summands than the former series.\\
\endproof

\begin{lemma}\label{DJ-CShift_k}\ \\
Let $({M_n})_{n\in\NN_0}$ be a sequence of positive real numbers. Then, for any $k\in\NN_0$,
$
\sum\limits_{n=1}^{\infty} \frac{M_{n+k-1}}{M_{n+k}} = \infty
$ if and only if  
$
\sum\limits_{n=1}^{\infty} \frac{M_{n-1}}{M_{n}} = \infty.
$
\end{lemma}

\begin{proof}\ \\
These two series differ only by a finite number of positive summands.\\ 
\end{proof}

\begin{lemma}\label{LemmadelloShift}\ \\
Let $({M_n})_{n\in\NN_0}$ be a log-convex sequence of positive real numbers. $C\{M_n\}$ is quasi-analytic if and only if for some (and hence for any) $k \in \NN_0$ the class $C\{M_{n+k}\}$ is quasi-analytic.
\end{lemma}
\proof\ \\
By Theorem \ref{DJ-C} and Lemma \ref{DJ-CShift_k}, $C\{M_n\}$ is quasi-analytic if and only if 
\begin{equation}\label{serieFrazioni_k}
\sum_{n=1}^\infty \frac{M_{n+k-1}}{M_{n+k}}=\infty.
\end{equation}
Note that the sequence
$\left(M_{n+k}\right)_{n\in\NN_0}$ is also log-convex. Hence, by Theorem~\ref{DJ-C}, \eqref{serieFrazioni_k} is equivalent to the quasi-analyticity of the class $C\left\{M_{n+k}\right\}$.\\
\endproof

\begin{theorem}\label{CarlShift}\ \\
Let $({M_n})_{n\in\NN_0}$ be a log-convex sequence of positive real numbers. If we have
$
\sum\limits_{n=1}^\infty \frac{1}{\sqrt[2n]{M_{2n}}} = \infty
$,
then 
$
\sum\limits_{n=1}^\infty \frac{1}{\sqrt[2n]{M_{2n+h}}} = \infty$ for any $h\in\NN_0$.
\end{theorem}
\proof\ \\
By Proposition~\ref{ConstTimesCarl}, we can assume w.l.o.g.\! that $M_0=1$. Let us consider separately the cases when $h$ is even or odd.\\
If $h$ is even, then the conclusion directly follows by applying Lemma~\ref{LemmadelloShift}.\\
If $h$ is odd, then we need some more considerations. Let us first note that for bounded $({M_n})_{n\in\NN_0}$, the result is obvious. Suppose $({M_n})_{n\in\NN_0}$ diverges, then there exists $N\in\NN$ such that for any $n\geq N$ we have that $M_n\geq 1$. Hence, for any $n\geq N$, we get that
\begin{equation}\label{(i)}
\frac{1}{\sqrt[2n-1]{M_{2n-1}}}\leq \frac{1}{\sqrt[2n]{M_{2n-1}}}.
\end{equation}
Moreover, by Corollary \ref{corInc}, we have that 
$
\frac{1}{\sqrt[2n]{M_{2n}}}\leq \frac{1}{\sqrt[2n-1]{M_{2n-1}}},$ for any $n\in\NN$. Hence, since $\sum\limits_{n=1}^\infty \frac{1}{\sqrt[2n]{M_{2n}}} = \infty$, also $\sum\limits_{n=1}^\infty \frac{1}{\sqrt[2n-1]{M_{2n-1}}} = \infty$. This together with \eqref{(i)} gives that
\begin{equation}\label{SerieStar}
\sum_{n=1}^\infty \frac{1}{\sqrt[2n]{M_{2n-1}}} = \infty.
\end{equation}
Let us consider the sequence $B_n:=\sqrt{M_{2n-1}}$ for $n\geq 1$. The log-convexity of $(M_n)_{n\in\NN_0}$ implies that $({B_n})_{n\in\NN}$ is also log-convex. Then \eqref{SerieStar} is equivalent to the quasi-analyticity of the class $C\{B_n\}$ by Theorem~\ref{DJ-C}. Let $k\in\NN$ be such that $h=2k-1$, then by applying Lemma \ref{LemmadelloShift} to the sequence $({B_n})_{n\in\NN}$, we get that $C\{B_{n+k}\}$ is quasi-analytic which proves our conclusion.
\\
\endproof

\section*{Acknowledgments}
The author would like to thank Tobias Kuna and Konrad Schm\"udgen for helpful discussions and comments on the manuscript.

\def\cprime{$'$}

\end{document}